\documentclass[10pt,a4paper]{amsart}
\usepackage{geometry}

\usepackage{etoolbox}
\apptocmd{\sloppy}{\hbadness 10000\relax}{}{}

\newcommand{\comment}[1]{}

\usepackage{microtype}
\usepackage{babel}
\usepackage[utf8]{inputenc}
\usepackage{enumitem}
\usepackage{amsmath,amsthm,amssymb,amsfonts}
\usepackage{tikz-cd}
\usetikzlibrary{arrows}
\input xy
\xyoption{all}
\newdir{ >}{{}*!/-5pt/@{>}}

\newcommand{\C}{\mathbb{C}}

\newcommand{\Q}{\mathbb{Q}}
\newcommand{\Z}{\mathbb{Z}}
\newcommand{\N}{\mathbb{N}}

\newcommand{\cA}{\mathcal A}
\newcommand{\cO}{\mathcal O}
\newcommand{\cR}{\mathcal R}
\newcommand{\cS}{\mathcal S}
\newcommand{\cV}{\mathcal V}

\newcommand{\topdf}{\texorpdfstring}

\newcommand{\BF}{\mathfrak{BF}}

\newcommand{\xto}{\xrightarrow}
\newcommand{\iso}{\overset{\sim}{\longrightarrow}}
\DeclareMathOperator{\reg}{reg}
\DeclareMathOperator{\colim}{colim}
\DeclareMathOperator{\coker}{coker}
\DeclareMathOperator{\idem}{idem}
\DeclareMathOperator{\id}{id}

\DeclareMathOperator{\can}{can}

\DeclareMathOperator{\sink}{sink}
\DeclareMathOperator{\rk}{rk}
\DeclareMathOperator{\ad}{ad}
\def\gr{\operatorname{gr}}

\newcommand{\tmap}{\iota}
\newcommand{\cat}[1]{\mathsf{#1}}
\newcommand{\hltimes}{{ \ \widehat{\ltimes}\ }}
\newcommand{\gBF}{\BF_{\gr}}

\usepackage{xcolor}
\definecolor{thmcol}{RGB}{9, 115, 51}
\definecolor{citecol}{RGB}{9, 115, 51}
\definecolor{linkcol}{RGB}{9, 115, 51}
\definecolor{urlcol}{RGB}{9, 115, 51}

\numberwithin{equation}{section}

\theoremstyle{plain}
\newtheorem*{thm*}{Theorem}
\newtheorem{thm}[equation]{Theorem}
\newtheorem{lem}[equation]{Lemma}
\newtheorem{coro}[equation]{Corollary}
\newtheorem{prop}[equation]{Proposition}

\newtheorem{ques}[equation]{Question}

\theoremstyle{definition}
\newtheorem{defn}[equation]{Definition}
\newtheorem{conj}[equation]{Conjecture}
\newtheorem{ex}[equation]{Example}

\theoremstyle{remark}
\newtheorem{rmk}[equation]{Remark}

\newtheorem{nota}[equation]{Notation}
\newtheorem*{ack}{Acknowledgements}
\newtheorem*{note}{Note}
\usepackage[colorlinks=true, pagebackref=true]{hyperref}
\hypersetup{
    colorlinks,
    citecolor=citecol,
    linkcolor=linkcol,
    urlcolor=urlcol
}
\usepackage{amsrefs}
\usepackage{url}

\title[]{Lifting morphisms between graded Grothendieck groups 
of Leavitt path algebras}
\author[1]{Guido Arnone}
\email{garnone@dm.uba.ar}
\address{Departamento de Matem\'atica-IMAS\\ Facultad de Ciencias Exactas y Naturales\\
Universidad de Buenos Aires, Argentina}
\thanks{Declaration of interest: none}
\date{}

\begin{document}

\maketitle

\begin{abstract} We show that any pointed, preordered module map
$\gBF(E) \to \gBF(F)$
between Bowen-Franks modules of finite graphs can be 
lifted to a unital, graded, diagonal 
preserving $\ast$-homomorphism 
$L_\ell(E) \to L_\ell(F)$ between 
the corresponding Leavitt 
path algebras over any 
commutative unital ring 
with involution $\ell$. Specializing to the case when $\ell$ is a field,
we establish the fullness part of Hazrat's conjecture 
about the functor from Leavitt path $\ell$-algebras
of finite graphs to preordered 
modules with order unit that maps $L_\ell(E)$
to its graded Grothendieck group. 
Our construction of lifts is of
combinatorial nature; we
characterize the maps arising 
from this construction as the
scalar extensions along $\ell$ of
unital, graded $\ast$-homomorphisms 
$L_\Z(E) \to L_\Z(F)$ 
that preserve a sub-$\ast$-semiring
introduced here.

\end{abstract}

\section{Introduction}

Let $E$
be a directed graph, and let $L_\ell(E)$ be its 
associated Leavitt path algebra
over a commutative ring $\ell$ with 
involution (\cite[Definition 2.5]{conmlpa}). This is 
a $\ast$-algebra which is 
graded over $\Z$. All graphs considered 
will have finitely many vertices and edges. 
We write $K_0^{\gr}(L_\ell(E))$ for the graded Grothendieck 
group of $L_\ell(E)$, that is, the group completion 
of the monoid of finitely generated projective $\Z$-graded $L_\ell(E)$-modules.
This is a preordered module with order unit $[L_\ell(E)]$; 
see Subsections \ref{subsec:preo} and \ref{subsec:gbf} for 
definitions of these terms and further details.
In \cite{hazrat} Hazrat conjectures that, when $\ell$ is a field, the graded
Grothendieck group classifies Leavitt path algebras as graded algebras.

\begin{conj}[{\cite[Conjecture 1]{hazrat}}]\label{conj:hazrat-iso}
Suppose that $\ell$ is a field.
Given graphs $E$ and $F$, the algebras $L_\ell(E)$ and $L_\ell(F)$ 
are graded isomorphic if and only if there 
is a preordered module 
isomorphism $K_0^{\gr}(L_\ell(E))\iso K_0^{\gr}(L_\ell(F))$ 
mapping $[L_\ell(E)]$ to $[L_\ell(F)]$.
\end{conj}

\begin{conj}[{\cite[Conjecture 3]{hazrat}}]\label{conj:hazrat-fullfaith}
Let $\ell$ be a field. The graded Grothendieck group 
is a fully faithful functor
from the category of unital 
Leavitt path $\ell$-algebras 
with graded homomorphisms modulo 
inner-automorphisms to the category 
of preordered abelian
groups with order unit.
\end{conj}

In the same article, Hazrat proves Conjecture \ref{conj:hazrat-iso}
for a class of graphs called polycephaly 
graphs (\cite[Definition 3.6]{hazrat}). A weaker version of 
Conjecture \ref{conj:hazrat-iso} was proven by 
Ara and Pardo in \cite{towards} for finite graphs with no sinks or sources. 
They also show that the faithfulness part 
of Conjecture \ref{conj:hazrat-fullfaith}
fails to hold in full 
generality (\cite[Example 6.7]{towards}).

The objective of this article is to address
the fullness part of Conjecture \ref{conj:hazrat-fullfaith}. In 
the particular case of a field, 
our main result reads as follows.

\begin{thm}\label{thm:mainfield} Let $\ell$ be a field. 
Given finite graphs $E$ and $F$ and $\phi \colon K_0^{\gr}(L_\ell(E))
\to K_0^{\gr}(L_\ell(F))$ a morphism of preordered modules such that
$\phi([L_\ell(E)])=[L_\ell(F)]$, there exists 
a unital, diagonal preserving
$\Z$-graded $\ast$-homomorphism
$\varphi \colon L_\ell(E) \to L_\ell(F)$ 
such that $K_0^{\gr}(\varphi) = \phi$. 
\end{thm}

Here diagonal preserving means that the map $\varphi$
above sends a distinguished commutative subalgebra of $L_\ell(E)$, 
the diagonal subalgebra, to that of $L_\ell(F)$. 

For an arbitrary commutative ring with involution $\ell$, 
we turn our attention to the Bowen-Franks module $\gBF(E)$
of a graph $E$, a notion closely related to that of the graded 
Grothendieck group. There are a distinguished 
element $1_E \in \gBF(E)$ and a comparison map
\[
\can \colon \gBF(E) \to K_0^{\gr}(L_\ell(E)), \qquad 1_E \mapsto [L_\ell(E)],
\]
which is an isomorphism whenever 
$K_0(\ell) = \Z$ (\cite[Corollary 5.4]{arcor}). Theorem \ref{thm:mainfield}
is then implied by the more general statement below.

\begin{thm}[Theorem \ref{thm:gbf-lift}] \label{thm:main} 
Let $E$ and $F$ be finite graphs. If
$\phi \colon \gBF(E)\to\gBF(F)$  is
a morphism of preordered modules mapping $1_E \mapsto 1_F$, 
then there exists a unital, $\Z$-graded, diagonal preserving 
$\ast$-homomorphism
$\varphi \colon L_\ell(E) \to L_\ell(F)$ such that
the following diagram commutes.
\begin{center}
\begin{tikzcd}
 K_0^{\gr}(L_\ell(E)) \arrow{r}{K_0^{\gr}(\varphi)} & K_0^{\gr}(L_\ell(F))\\
 \gBF(E) \arrow{u}{\can}\arrow{r}{\phi} & \arrow{u}[right]{\can} \gBF(F)
\end{tikzcd}
\end{center}
\end{thm}

A natural question that arises from the theorem above
is whether the lift of an
isomorphism between Bowen-Franks modules 
is necessarily an algebra isomorphism. 
Using recent results of Carlsen, Dor-On and Eilers \cite{shift},
we relate this to a question on the associated graph $C^\ast$-algebras
(Question \ref{ques:norefle}), which can also be interpreted 
as a question 
in terms of shift equivalences and strong shift equivalences
of matrices.
In Theorem \ref{thm:norefle} we show
that an affirmative answer to this question would imply that, for Leavitt path $\C$-algebras, the functor $K_0^{\gr}$
does not reflect isomorphisms.
We also note that a counterexample to the Williams 
conjecture by Kim and Roush \cite{KR} yields 
a pair of graphs satisfying almost all of 
the conditions in Question \ref{ques:norefle}; see
Question \ref{ques:kr} and Remark \ref{rmk:kr} for further details.

The construction of lifts given in
Theorem \ref{thm:main} is of combinatorial nature;
we translate the information encoded by a 
Bowen-Franks module map 
to the existence of certain partitions of paths in 
a graph and bijections between them. We call the 
homomorphisms that arise in such a way \emph{tidy};
see Definition \ref{defn:tidy} for a precise statement.
We prove a characterization of tidy maps in terms 
of a sub-$\ast$-semiring 
of a Leavitt path $\Z$-algebra introduced here, which we call its \emph{positive cone} (see Definition \ref{def:pc}).
A morphism $L_\Z(E) \to L_\Z(F)$ will be said to be
order preserving if it maps the positive cone 
of $L_\Z(E)$ to that of $L_\Z(F)$. 
The theorem below lets us conclude, 
in particular, that the composite of two tidy homomorphisms
is again tidy (Corollary \ref{coro:tidy-comp}).

\begin{thm}[Theorem \ref{thm:tidy-char}]\label{thm:introtidy}
Let $E$ and $F$ be finite graphs
and $\varphi \colon L_\ell(E) \to L_\ell(F)$
a unital $\ell$-algebra homomorphism. 
The following statements are equivalent.
\begin{itemize}
    \item[i)] The morphism $\varphi$ is tidy.
    \item[ii)] The morphism $\varphi$ is the scalar extension along $\ell$ 
    of a unital, order preserving, $\Z$-graded 
    $\ast$-homomorphism $L_\Z(E) \to L_\Z(F)$.
\end{itemize}
\end{thm}

The rest of the article is organized as follows. In Section \ref{sec:prelim}
we recall the basic notions of 
Leavitt path algebras and Bowen-Franks modules.
Section \ref{sec:gr-mvn} is dedicated to auxiliary results
regarding the relation between the graded Grothendieck 
group and homogeneous idempotents. In Section \ref{sec:gbf}
we establish some properties of Bowen-Franks modules
using a presentation introduced in \cite{arcor}; this is then 
used in Section \ref{sec:maps-gbf} to give a characterization 
of maps between Bowen-Franks modules in terms of
vertex indexed matrices with nonnegative integer coefficients (Proposition \ref{prop:bf-map-mat}). 
With this in place, in 
Section \ref{sec:lift} we prove Theorem \ref{thm:main}
as Theorem \ref{thm:gbf-lift}. Theorem \ref{thm:introtidy} 
is proven in Section \ref{sec:tidyequiv} as Theorem \ref{thm:tidy-char}.

\begin{note}
We point out that, simultaneously and independently,
Lia Va\v{s} has proved a lifting 
result \cite[Theorem 3.2]{vas} in the case when $\ell$ is a field, 
similar to Theorem \ref{thm:mainfield}.
The lifts obtained in loc. cit. are not shown
to be involution nor diagonal preserving.  
We also 
remark that \cite[Theorem 3.2]{vas} 
holds for graphs with finitely many vertices and countably many edges.
In \cite[Section 5.2]{vas}, some non-constructive steps 
of the proof of \cite[Theorem 3.2]{vas} are discussed, and 
the question of whether one can explicitly produce
a lift is posed. We note that the proofs 
of Proposition \ref{prop:bf-map-mat} and  
Theorem \ref{thm:gbf-lift}
give an algorithm to construct lifts; however,
we exhibit no bound for the finitely many steps 
needed to obtain
Equations \eqref{eq:linearity} and \eqref{eq:unitality}.
We would also like to remark that 
the implications 
$(1) \Rightarrow (5^+) \Rightarrow (5)$ of \cite[Theorem 4.1]{vas}
can be recovered by setting $L=1$ in Corollary \ref{coro:inner-diag}.
\end{note}

\begin{nota} In this article the natural numbers do not include zero; we will write $\N := \Z_{\ge 1}$
and $\N_0 := \N \cup \{0\}$.
The letter $\sigma$ will denote the generator of the infinite cyclic group $C_\infty \simeq \Z$,
written multiplicatively. Its group ring will be denoted $\Z[\sigma] := \Z[C_\infty]$.
\end{nota}

\begin{ack} I wish to thank Guillermo Corti\~nas for his
advice and encouragement to write this article, as well as for
his valuable comments and suggestions. 
This research was supported by a CONICET doctoral fellowship and partially supported by grants
PICT 2017-1395 from ANPCyT,
PIP 2021-2023 GI, 00423CO from CONICET and
UBACyT 256BA from UBA.
\end{ack}

\section{Preliminaries} \label{sec:prelim}

\numberwithin{equation}{subsection}

\subsection{Graphs}

A \emph{graph} $E$ consists of two sets $E^0$ of \emph{vertices} and $E^1$ of \emph{edges},
together with source and range functions $s,r \colon E^1 \to E^0$.
All graphs in this paper will be assumed to be finite, meaning that
$E^0$ and $E^1$ are finite sets.

A vertex $v \in E^0$ is \emph{regular} if $s^{-1}(v) \neq \emptyset$
and a \emph{sink} otherwise. The set of regular vertices will be denoted
$\reg(E) \subset E^0$; its complement is
$\sink(E) = E^0 \setminus \reg(E)$.
We say that $E$ is regular if $E^0 = \reg(E)$.  

Given $v,w \in E^0$, we will write $E_{v,w} = s^{-1}(v) \cap r^{-1}(w)$
for the set of edges with source $v$ and range $w$.
The \emph{adjacency matrix} of a graph $E$ is
the matrix $\cA_E \in \Z^{E^0 \times E^0}$,
whose $(v,w)$-entry is the amount
of edges in $E$ with source $v$ and range $w$, $(\cA_E)_{v,w} = \#E_{v,w}$.
The \emph{reduced adjacency matrix}
$A_E \in \Z^{\reg(E) \times E^0}, (A_E)_{v,w} = (\cA_E)_{v,w}$, is the matrix
obtained from removing the rows of $\cA_E$ corresponding to sinks.

A \emph{path} in $E$ is a finite
sequence of edges $\alpha = e_1 \ldots e_n$ such that
$r(e_i) = s(e_{i+1})$ for each $i \in \{1, \ldots, n-1\}$. The
\emph{source} of $\alpha$ is $s(\alpha) =  s(e_1)$ and its range
is $r(\alpha) = r(e_n)$.
The \emph{length} of $\alpha$ is $|\alpha| = n$. We will consider a vertex
$v \in E^0$ as a path of length zero with source and range $v$.
The set of paths of $E$ will be
denoted $E^\infty$. For each $k \geq 0$ and $v,w \in E^0$, we shall write
$E^k_{v,w} = \{\alpha \in E^\infty : s(\alpha) = v, r(\alpha) = w, |\alpha| = k\}$ and $E^k_w := \bigcup_{v \in E^0} E^k_{v,w} = \{\alpha \in E^\infty : r(\alpha) = w, |\alpha| = k\}$. For convenience,
we recall that $\#E^k_{v,w} = (\cA^k)_{v,w}$ and thus 
\begin{equation}\label{eq:powers-paths}
    \#E^k_{w} = \sum_{v \in E^0} (\cA_E^k)_{v,w}.
\end{equation}
We shall also need to distinguish between paths 
of a certain length ending at sinks or regular vertices; given
$n \ge 0$ we put $\mathcal R_n(E) = \bigcup_{w \in \reg(E)} E^n_{w}$
and $\mathcal S_n(E) = \bigcup_{u\in \sink(E)}\bigcup_{i = 0}^n E^i_{u}$.

\subsection{Leavitt path algebras}

Throughout the paper, we fix a commutative
unital ring $\ell$ equipped with an involution
$\ast \colon \ell \to \ell$.

Let $E$ be a graph
and $L(E) := L_\ell(E)$ its Leavitt path algebra
over $\ell$ \cite[Definition 2.5]{conmlpa}
equipped with its canonical $\Z$-grading
\cite[Proposition 4.7]{conmlpa}. 
Recall that by \cite[Theorem 8.1]{conmlpa} there is an
isomorphism $L(E) = \ell \otimes_\Z L_\Z(E)$
and, in particular, when $\ell$ is a field this definition
agrees with that of \cite[Definition 1.2.3]{lpabook}.
We consider the canonical involution on $L_\Z(E)$ as defined in
\cite[Remark 4.1]{conmlpa}
and equip $L(E)$ with the tensor product involution.

Unless specified otherwise, whenever we refer to \cite{lpabook},
we shall implicitly mean that the cited argument holds true for a Leavitt path
algebra over any commutative unital ring.
We record the following lemmas for future use.

\begin{lem}\label{lem:pathsort} Let $\alpha$ and $\beta$
be two paths in a finite graph $E$. The following identities hold in $L(E)$:
\begin{itemize}
    \item[(i)] If $|\alpha| = |\beta|$, then $\alpha^\ast \beta = \delta_{\alpha, \beta} r(\alpha)$.
    \item[(ii)] If $r(\alpha) \in \reg(E), r(\beta) \in \sink(E)$ and $|\beta| \leq |\alpha|$,
    then $\alpha^\ast \beta = \beta^\ast \alpha = 0$.
    \item[(iii)] If $r(\alpha), r(\beta) \in \sink(E)$,
    then $\alpha^\ast \beta = \delta_{\alpha, \beta} r(\alpha)$.
    \item[(iv)] For each $N \geq 0$, the set $\{xx^\ast : x\in \cR_N(E) \sqcup \cS_N(E)\}$
    consists of homogeneous orthogonal projections.
\end{itemize}
\end{lem}
\begin{proof}Straightforward from \cite[Lemma 1.2.12]{lpabook}.
\end{proof}

\begin{lem}\label{lem:1-paths} Let $E$ be a finite graph. For each $N \in \N_0$,
\[
1 = \sum_{\alpha \in \cR_N(E)} \alpha \alpha^\ast
+ \sum_{\beta \in \cS_N(E)}\beta\beta^\ast
\]
\end{lem}
\begin{proof} The lemma is true when $N = 0$ by \cite[Lemma 1.2.12 (iv)]{lpabook}. The general case
follows inductively from \cite[Relation (CK2) in Definition 1.2.3]{lpabook}.
\end{proof}

Recall that given a graph $E$, the \emph{diagonal} 
of $L(E)$ is the sub-$\ast$-algebra 
$D(E) = \mathrm{span}_\ell\{\alpha\alpha^\ast : \alpha \in E^\infty\} \subset L(E)_0$. An algebra homomorphism $\varphi \colon L(E) \to L(F)$ is \emph{diagonal preserving} if $\varphi(D(E)) \subset D(F)$.
Writing
\[
L(E)_{0,n} = \mathrm{span}_\ell\{\alpha\beta^\ast  : \alpha, \beta \in \mathcal R_n(E)\}
\oplus \mathrm{span}_\ell\{\alpha\beta^\ast  : \alpha,
\beta \in \mathcal S_n(E)\}
\]
and
\[
D(E)_n =  \mathrm{span}_\ell\{\alpha\alpha^\ast  : \alpha \in \mathcal R_n(E) \sqcup \mathcal S_n(E)\},
\]
each algebra $L(E)_{0,n}$ is matricial and $D(E)_n$ is its diagonal subalgebra. We have increasing unions 
$L(E)_0 = \bigcup_{n \ge 0} L(E)_{0,n}$, $D(E) = \bigcup_{n \ge 0} D(E)_n$.

\subsection{Preordered \topdf{$\Z[\sigma]$}{Z[σ]}-modules}
\label{subsec:preo}

A \emph{preordered $\Z[\sigma]$-module} is a $\Z[\sigma]$-module $M$ together with an
additive submonoid $M_+$ such that $\sigma M_+ \subset M_+$ and $\sigma^{-1} M_+ \subset M_+$.
For $m, n \in M$ , we will write $m \geq n$ to
mean that $m-n \in M_+$. An \emph{order unit} in a preordered module $M$ is an element
$u \in M$ such that for all $m \in M$, there exist $x \in \N_0[\sigma]$ such that $xu \ge m$.
A morphism of preordered $\Z[\sigma]$-modules with order unit is a $\Z[\sigma]$-linear
map $f \colon (M,u) \to (N,v)$ such that $f(M_+)\subset N_+$ and $f(u) = v$.

\subsection{Bowen-Franks modules} 
\label{subsec:gbf}
Let $E$ be a graph and $I \in \Z^{E^0 \times \reg(E)}$
the matrix defined by $I_{v,w} = \delta_{v,w}$. The \emph{Bowen-Franks $\Z[\sigma]$-module} of
$E$ is
\[
\gBF(E) := \coker(I-\sigma A_E^t) = \frac{\Z[\sigma]^{E^0}}{\langle v
- \sigma \sum_{e \in s^{-1}(v)} r(e) : v \in \reg(E)\rangle}
\]
There is a canonical preordered module structure
on $\gBF(E)$ given by the submonoid $\gBF(E)_+$ generated by the elements
$\sigma^i v$ for each $v \in E^0$ and $i\in \Z$ (this is 
known as the \emph{talented monoid} of $E$; see \cite{talented}). 
The element $1_E := \sum_{v \in E^0} v$
is an order unit with respect to this preordered module structure.

Write $K_0^{\gr}(L(E))$ for the group completion of the monoid of projective,
finitely generated, graded $L(E)$-modules.
This group is equiped with
a canonical preordered $\Z[\sigma]$-module structure.
For each homogeneous
idempotent $x \in L(E)_0$, we define $[x] := [L(E)x] \in K^{\gr}_0(L(E))$.
There is a canonical preordered $\Z[\sigma]$-module homomorphism
\begin{equation}\label{def:can}
    \can \colon \gBF(E) \to K_0^{\gr}(L(E)), \quad v \mapsto [v].
\end{equation}
By \cite[Corollary 5.4]{arcor} the map above is an isomorphism whenever
the Grothendieck group of $\ell$ is isomorphic to $\Z$; for example,
such is the case when $\ell$ is a field or more generally a PID.

\begin{rmk}Let $E$ be a graph. By \cite[Lemma 5.1]{arcor}, the following sequence 
is exact:
\[
0\to \Z[\sigma]^{\reg(E)} \xto{I-\sigma A_E^t} \Z[\sigma]^{E^0} \to \gBF(E) \to 0.
\]
Thus, writing $K = \Q(\sigma)$ for the fraction field of $\Z[\sigma]$, 
it follows that 
\begin{equation}\label{eq:rk-bfg}
\rk \gBF(E) := \dim_K K \otimes_{\Z[\sigma]} \gBF(E) = \# \sink(E).
\end{equation}
\end{rmk}

\numberwithin{equation}{section}

\section{Graded idempotents and Murray-von Neumann equivalence}\label{sec:gr-mvn}

Let $R$ be a ring. Recall that the idempotent elements $\idem(R)$ of $R$ are partially 
ordered by defining $e \le f$ whenever $ef = fe  = e$ and that $e,f \in \idem(R)$
are said to be \emph{Murray-von Neumann equivalent}, denoted $e \sim f$, if there exist $x,y \in R$ 
such that $xy = e$, $yx = f$. 
This is an equivalence relation, put $\cV(R) = \idem(R)/\sim$ and 
$\cV_\infty(R) := \idem(M_\infty R)/\sim$. We also 
recall that the block sum of matrices makes $\cV_\infty(R)$ into a commutative monoid.

\begin{rmk} \label{rmk:vgr}
Recall that a ring $R$ has
\emph{local units}  if 
for each finite subset $F \subset R$ there exists 
$e \in \idem(R)$ satisfying $F \subset eRe$. If $R$ 
is a ring with local units, then
by \cite[Section 4A]{stein} the monoid $\cV_\infty(R)$ 
is isomorphic to the monoid of isomorphism 
classes of finitely generated projective unital $R$-modules. 
Let $S$ be a unital $\Z$-graded ring and let $\Z \hltimes S$ be 
its crossed product as defined in \cite[Subsection 2.5]{arcor}.
By \cite[Section 2C]{stein} (see also \cite[Section 3.1]{arcor}), 
we have that 
$\cV_\infty(\Z \hltimes S)$ is naturally isomorphic to the 
monoid $\cV^{\gr}(S)$ of isomorphism classes of finitely 
generated $\Z$-graded projective $S$-modules.  
In particular $K_0^{\gr}(S)$ is the group
completion of $\cV_\infty(\Z \hltimes S)$.
\end{rmk}

\begin{lem}\label{lem:mvn-corner} 
Let $R$ be a ring and $e,f,g \in \idem(R)$ such that $e,f \leq g$. The 
following statements are equivalent:
\begin{itemize}
\item[i)] There exist $x,y \in gRg$ such that $xy = e$, $yx = g$.
\item[ii)] The idempotents $e$ and $f$ are Murray-von Neumann equivalent. 
\end{itemize}
\end{lem}
\begin{proof} Since $gRg \subset R$, we have that 
$i) \Rightarrow ii)$. To prove $ii) \Rightarrow i)$, 
consider $x,y \in R$ such that $xy = e$, $yx = f$
and put $x' = exf$, $y' = fye$. Then $x'y' = e$, 
$y'x' = f$ and $x' \in eRf = geRfg \subset gRg$, $y \in gRg$.
\end{proof}

\begin{lem} \label{lem:v-cor-mono} 
If $R$ is a ring and $g \in \idem(R)$, then the inclusion induced map 
$\cV(gRg) \to \cV_\infty(R)$ is injective.
\end{lem}
\begin{proof} 
Write $\iota_k \colon gRg \to M_k R$ for the upper-left corner inclusion. 
Let $e,f \in gRg$ be two idempotents that are Murray-von Neumann equivalent as elements of $M_\infty R$. There exists
$n \ge 1$ such that $\iota_n(e) \sim \iota_n(f)$ in $M_n R$. 
Since $e,f \le g$,  it follows that $\iota_n(e),\iota_n(f) \leq \iota_n(g)$; by Lemma \ref{lem:mvn-corner}
there must exist $a,b \in \iota_n(g)(M_n R)\iota_n(g) = \iota_n(gRg)$ such that 
$ab = \iota_n(e)$, $ba = \iota_n(f)$. Let $x,y \in gRg$ be such that 
$a = \iota_n(x), b = \iota_n(y)$. Since $\iota_n$ is injective, it follows 
that $xy = e$, $yx = f$, thus proving that $e \sim f$.
\end{proof}

\begin{coro} \label{coro:deg0-inj} If $R$ is a $\Z$-graded unital ring, the canonical map
$\cV(R_0) \to \cV_\infty(\Z \hltimes R)$ is injective.
\end{coro}
\begin{proof} Apply Lemma \ref{lem:v-cor-mono} to $\Z \hltimes R$ and the idempotent 
$g = \chi_0 \hltimes 1$, and notice that $R_0 \simeq g(\Z \hltimes R)g$
via $r \mapsto \chi_0 \hltimes r$.
\end{proof}

\begin{coro} \label{coro:inner-diag}
Suppose that $\ell$ is a field. Let $E$ and $F$ be finite graphs and 
$f,g \colon L(E) \to L(F)$ two 
unital $\Z$-graded maps. If $K_0^{\gr}(f) = K_0^{\gr}(g)$, then for each $L \ge 0$
there exists a degree zero unit $u_L \in L(E)$ such 
that $\ad(u_L) \circ f$ and $g$ 
coincide on $D(E)_L = \mathrm{span}_\ell\{xx^\ast \colon x \in \cR_L \sqcup \cS_L\}$.
\end{coro}
\begin{proof} Since $\cV^{\gr}(L(F))$ is a cancellative monoid when $\ell$ is a field by \cite[Corollary 5.8]{stein}, using Remark \ref{rmk:vgr} and
Corollary \ref{coro:deg0-inj} we get that the map $\cV(L(F)_0) \to K_0^{\gr}(L(F))$
is injective. 
Since $[f(xx^\ast)] = [g(xx^\ast)]$ in $K_0^{\gr}(L(F))$ for each
$x \in X := \cR_L \sqcup \cS_L$, there exist $p_x \in f(xx^\ast)L(F)_0 g(xx^\ast)$, 
$q_x \in g(xx^\ast) L(F)_0 f(xx^\ast)$ such that $p_xq_x = f(xx^\ast)$, 
$q_x p_x = g(xx^\ast)$. Now put $u_L := \sum_{x \in X}q_x$
and proceed as in \cite[Lemma 9.7]{kkhlpa}.
\end{proof}

\begin{lem} \label{lem:paths-gbf}
Let $E$ be a finite graph and $\can$ as in \eqref{def:can}. 
If $\alpha$ is a path of length
$N \in \N_0$ in $E$, then
\[
\can(\sigma^N r(\alpha)) =
[\alpha\alpha^\ast].
\]
\end{lem} 
\begin{proof}  Via the identification $K_0^{\gr}(L(E)) = \cV_\infty(\Z \hltimes L(E))^+$, the element $[L(E)\alpha\alpha^\ast] \in K_0^{\gr}(L(E))$ corresponds to the class  of the idempotent
$\chi_0 \hltimes \alpha \alpha^\ast \in \Z \hltimes L(E)$.
By \cite[Theorem 5.2]{arcor}, the
element $\can(\sigma^N r(\alpha))$ corresponds to the class
of the idempotent element $\chi_{n} \hltimes v \in \Z \hltimes L(E)$.
It suffices then to note that the elements $\chi_{0} \hltimes \alpha$,
$\chi_{n} \hltimes \alpha^\ast \in \Z \hltimes L(E)$ implement a Murray-von Neumann equivalence between $\chi_n \hltimes v$ and $\chi_0 \hltimes \alpha\alpha^\ast$.
\end{proof}

\section{A characterization of Bowen-Franks modules}\label{sec:gbf}

Let $E$ be a finite graph.
We first recall from \cite{arcor} a characterization of $\gBF(E)$ and establish
further properties of this presentation.
For each $k \in \Z$ and $v \in E^0$, write $v_k = (v,k)$.
Given $x = \sum_{v \in E^0} x_v v \in \Z^{E^0}$ we will
write
\[
x \otimes \sigma^k = \sum_{v \in E^0} x_v v_k \in \Z^{E^0 \times \{k\}}.
\]

Put
$V_n := \{u_i : u \in \sink(E), \ |i| \leq n\} \cup \{w_n : w \in \reg(E)\}$
for each $n \geq 0$ and define $\Z$-linear maps
$\tmap_n \colon \Z^{V_n} \to \Z^{V_{n+1}}$ via
\[
\tmap_n(u_i) := u_i, \qquad \tmap_n(w_n) := \sum_{v \in E^0} (A_E)_{w,v} v_{n+1}.
\]
The maps above form an $\N_0$-indexed filtered system of abelian groups. Given
two non-negative integers $i < j$, we will write
$\tmap_{i,j} = \tmap_{j-1} \circ \cdots \circ \tmap_{i}$ for the transition map. We also put $\iota_{i,i} = \id_{\Z^{V_i}}$.
The following proposition follows from \cite[proof of Theorem 5.2]{arcor}.

\begin{prop} \label{prop:gbf-colim}
Let $E$ be a finite graph. There is a $\Z[\sigma]$-module
isomorphism
\begin{equation}\label{eq:iso-gbf-colim}
    \gBF(E) \cong \colim (\Z^{V_0} \xto{\tmap_0} \Z^{V_1} \to \cdots),
    \qquad \sigma^n v \mapsto [v_n]
\end{equation}
that maps $\gBF(E)_+$ to
$\colim(\N_0^{V_0} \xto{\tmap_0} \N_0^{V_1} \to \cdots)$ and $1_E$ to
$\sum_{v \in E^0} [v_0]$.
\qed
\end{prop}

\begin{rmk}\label{rmk:filtercolim} Let $E$ be a finite graph.
Since $\gBF(E)$ is a filtering
colimit, it follows that if $x \in \Z^{V_n}$ and $y \in \Z^{V_m}$ are such that
$[x] = [y]$ in $\gBF(E)$ then there exists $k_0 \geq n,m$ such that
$\tmap_{n,k}(x) = \tmap_{m,k}(y)$ for all $k \geq k_0$.
Moreover; if $x_1 \in \Z^{V_{n_1}}, \ldots, x_j \in \Z^{V_{n_j}},
y_1 \in \Z^{V_{m_1}}, \ldots, y_j \in \Z^{V_{m_j}}$ are such that
$[x_i] = [y_i]$ for all $i \in \{1, \ldots, j\}$, there exists
$k_0 \geq n_1, \ldots, n_j, m_1, \ldots, m_j$ such that $\tmap_{n_i, k}(x_i)
= \tmap_{n_j, k}(y_i)$ for all $1 \leq i \leq j$ and $k \geq k_0$.
\end{rmk}

\begin{defn} \label{defn:BE-CE}
Let $E$ be a finite graph. We will write $B_E \in \Z^{\reg(E) \times \reg(E)}$,
and $C_E \in \Z^{\sink(E) \times \reg(E)}$ for the matrices obtained
from projecting $A_E^t$ onto $\Z^{\reg(E)}$ and $\Z^{\sink(E)}$ respectively,
\[
(B_E)_{v,w} = (A_E)_{w,v}, \qquad (C_E)_{v,u} = (A_E)_{u,v} \qquad
(u \in \sink(E), \ v,w \in \reg(E)).
\]
In particular we have $A_E^t x = B_E x  + C_E x$ for each $x \in \Z^{\reg(E)}$
and if $s^{(-N)}, \ldots, s^{(N)} \in \Z^{\sink(E)}$, $r \in \Z^{\reg(E)}$ then
\begin{equation}\label{eq:s-vec}
    \tmap_N(\sum_{i=-N}^N s^{(i)} \otimes \sigma^i + r \otimes \sigma^N) =
\sum_{i=-N}^N s^{(i)} \otimes \sigma^i + A_E^t r \otimes \sigma^{N+1} =
\sum_{i=-N}^N s^{(i)} \otimes \sigma^i + C_E r \otimes \sigma^{N+1} + B_E r \otimes \sigma^{N+1}.
\end{equation}
\end{defn}

The following identities are straightforward
from \eqref{eq:s-vec}.

\begin{lem} \label{lem:iterating}
Let $E$ be a finite graph and $N \in \N$. Let
$s^{(-N)}, \ldots, s^{(N)} \in \Z^{\sink(E)}$,
$r \in \Z^{\reg(E)}$. If $x = \sum_{i=-N}^N s^{(i)} \otimes \sigma^i
+ r \otimes \sigma^N \in \Z^{V_N}$, then:
\begin{itemize}
    \item[(i)] $\sigma^{-1}[x] = \sum_{i=-N}^N [s^{(i)} \otimes \sigma^{i-1}] +
    [A_E^t r \otimes \sigma^N]$ in $\gBF(E)$.
    \item[(ii)] If $M \geq 0$, then
    \[
        \tmap_{N,N+M}(x) =
        \sum_{i=-N}^N s^{(i)}\otimes \sigma^i +
        \sum_{j=0}^{M-1} C_E B_E^{j}r \otimes \sigma^{N+j+1}
        + B_E^{M}r \otimes \sigma^{N+M}.
    \]
    \item[(iii)] If $M \geq 0$, then
    \[
        \tmap_{N,N+M}(A_E^t r \otimes \sigma^N) =
        \sum_{j=0}^{M} C_E B_E^{j}r \otimes \sigma^{N+j}
        + B_E^{M+1} r \otimes \sigma^{N+M}
    \]
\end{itemize}
\qed
\end{lem}

We now turn to characterizing the representatives
of the order unit of $\gBF(E)$ for a finite graph $E$.
For each $n\geq 0$, let
\begin{equation}\label{def:ordunit}
    \mathfrak{u}_n := \tmap_{0,n}(1_E) \in \Z^{V_n}.
\end{equation}

\begin{lem} \label{lem:colim-ordunit} Let $E$ be a finite graph. For each $N \in \N_0$,
\begin{equation}\label{eq:colim-ordunit}
\mathfrak{u}_N = \sum_{u \in \sink(E)} \sum_{i=0}^N \#E_{u}^i u_i +
\sum_{v \in \reg(E)} \#E_{v}^N v_N.
\end{equation}
\end{lem}
\begin{proof} Since
the lemma holds for $N = 0$ by definition of $\mathfrak{u}_0$ and $1_E$,
we may suppose that $N \geq 1$.

By Lemma \ref{lem:iterating} (ii), we have
to show that $\#E^i_{u} = (C_E B_E^{i-1} 1_{\reg(E)})_u$
for each $u \in \sink(E)$, $i \in \{1,\ldots, N\}$ and
$\#E^N_{v} = (B_E^N 1_{\reg(E)})_v$ for each $v \in \reg(E)$.

Observe that viewed as an endomorphism of $\Z^{\sink(E)} \oplus \Z^{\reg(E)}$, the
matrix $\cA_E^t$ has the form $\cA_E^t = \begin{pmatrix}
0 & C_E\\0 & B_E\end{pmatrix}$ and inductively,
\[
(\cA_E^t)^i = \begin{pmatrix}
0 & C_E B_E^{i-1}\\0 &B_E^i\end{pmatrix} \qquad (1 \leq i \leq N).
\]
To conclude we note that Equation \eqref{eq:powers-paths}
can be restated as $\#E^i_{w} = ((\cA_E^t)^i 1_E)_w$
for each vertex $w \in E^0$.
\end{proof}

\begin{lem}\label{lem:no-neg}Let $E$ be a finite graph and $N \in \N$. Let
$s^{(-N)}, \ldots, s^{(N)} \in \Z^{\sink(E)}$, $r \in \Z^{\reg(E)}$
and $x = \sum_{i = -N}^N  s^{(i)} \otimes \sigma^i +
\sum_{v \in \reg(E)} r \otimes \sigma^N$.
If $[1_E] = [x]$ in $\gBF(E)$, then $s^{(i)} = 0$ for all $i \in \{-N, \ldots, -1\}$.
\end{lem}
\begin{proof} By Remark \ref{rmk:filtercolim}, there exists $M \in \N$
such that $\tmap_{N,M}(x) = \tmap_{0,N+M}(1_E) = \mathfrak{u}_{N+M}$ in $\Z^{V_{N+M}}$.
Hence $\mathfrak{u}_{N+M} = \tmap_{N,M}(x) = \sum_{i = -N}^N s^{(i)} \otimes \sigma^i
+ \tmap_{N,M}(r \otimes \sigma^N)$.
Writing $P = V_{N+M} \setminus (\sink(E) \times \{-(N+M), \ldots, -1\})$, notice
that $\mathfrak{u}_{N+M}, \tmap_{N,M}(r \otimes \sigma^N)$ and $s^{(i)} \otimes \sigma^i$
for non-negative
$i$ belong to $\Z^{P}$.
Thus,
\[
    \sum_{i = -N}^{-1} s^{(i)} \otimes \sigma^i =
    \mathfrak{u}_{N+M} - \tmap_{N,M}(r \otimes \sigma^N)
    - \sum_{i = 0}^{N} s^{(i)} \otimes \sigma^i
    \in \Z^{P} \cap \Z^{V_{N+M} \setminus P} = 0.
\]
This completes the proof.
\end{proof}

\section{Maps between Bowen-Franks modules}\label{sec:maps-gbf}

The main result of this section is the following characterization
of morphisms between Bowen-Franks modules.

\begin{prop}\label{prop:bf-map-mat} Let $E$ and $F$ be a finite graphs and
$\phi \colon \gBF(E) \to \gBF(F)$ a morphism of preordered $\Z[\sigma]$-modules
such that $\phi(1_E) = 1_F$.

For each $L_0 \in \N_0$ 
there exist $L \ge L_0$ and
matrices $S^{(0)}, S^{(1)}, \ldots, S^{(L)} \in \N_0^{E^0 \times \sink(F)}$,
$R \in \N_0^{E^0 \times \reg(F)}$ such that, for each $v \in E^0$,
\begin{equation}\label{eq:bf-mor-mat}
    \phi([v]) = \sum_{u \in \sink(F)} \sum_{i=0}^L S^{(i)}_{v,u} [u_i] +
    \sum_{w \in \reg(F)}R_{v,w} [w_L].
\end{equation}
Moreover, for each $v \in \reg(E), w \in \reg(F), u \in \sink(F), i \in \{0,\ldots,L\}$,
we have the following equations:
\begin{align}
&\#F^L_{w} = \sum_{z \in E^0} R_{z,w}, \quad
\#F^i_{u} = \sum_{z \in E^0} S^{(i)}_{z,u},\label{eq:bf-mor-P}\\
&S^{(0)}_{v,u} = 0, \quad S^{(i)}_{v,u} =
(A_E S^{(i-1)})_{v,u} \ (1 \leq i \leq L), \quad
(A_E S^{(L)})_{v,u} = (RA_F)_{v,u},\label{eq:bf-mor-S}\\
&(A_E R)_{v,w} = (RA_F)_{v,w}.\label{eq:bf-mor-AER=RAF}
\end{align}
\end{prop}
\begin{proof} Since $\phi$ is order preserving and $E^0$ 
is finite, by \eqref{eq:iso-gbf-colim}
there exist $N\ge L_0$ 
and a family $(x_v)_{v\in E^0}$ with
$x_v \in \N_0^{V_{N}}$ such that $\phi([v]) = [x_v]$ for all $v\in E^0$.
Thus, for each $v \in E^0$
and $i \in \{-N, \ldots, N\}$, there exist vectors $s^{(-N)}_v, \ldots,
s^{(N)}_v \in \N_0^{\sink(F)}, \widetilde{r}_v \in \N_0^{\reg(F)}$ such that
\[
\phi([v]) = \sum_{i=-N}^N [s^{(i)}_v \otimes \sigma^i] + [\widetilde{r}_v \otimes \sigma^N].
\]
From here we see that
\begin{equation*}
[1_F] = \sum_{v \in E^0} \phi([v])  =
\sum_{i=-N}^N [\sum_{v \in E^0} s^{(i)}_v \otimes \sigma^i]
+ [\sum_{v \in E^0} \widetilde{r}_v \otimes \sigma^N].
\end{equation*}
By Lemma \ref{lem:no-neg}, this implies that $\sum_{v \in E^0} s^{(i)}_v = 0$ for
each $v \in E^0$ and negative $i$. Since the coefficients of each vector
$s^{(i)}_v$ are
non-negative, we obtain that $s^{(i)}_v = 0$
for all $v \in E^0, i \in \{-N, \ldots, -1\}$.
As a consequence, the equations above simplify to $\phi([v]) =
\sum_{i=0}^N [s^{(i)}_v \otimes \sigma^i] + [\widetilde{r}_v \otimes \sigma^N]$ and
\begin{equation}\label{eq:colim-unitality}
[\mathfrak{u}_N] = [1_F] = \sum_{v \in E^0} \phi([v])  =
\sum_{i=0}^N [\sum_{v \in E^0} s^{(i)}_v \otimes \sigma^i]
+ [\sum_{v \in E^0} \widetilde{r}_v \otimes \sigma^N].
\end{equation}

Now, for each $v \in \reg(E)$ the equation  $\sigma^{-1}\phi([v]) =
\phi(\sigma^{-1}[v])$ and Lemma \ref{lem:iterating} (i)
yield
\begin{equation}\label{eq:colim-linearity}
    \sum_{i=0}^N [s^{(i)}_v \otimes \sigma^{i-1}]
    + [A_F^t \widetilde{r}_v \otimes \sigma^N] =
    \sum_{i=0}^N [\sum_{x \in E^0} (A_E)_{v,x} s^{(i)}_x \otimes \sigma^i] +
    [\sum_{x \in E^0} (A_E)_{v,x} \widetilde{r}_x \otimes \sigma^N].
\end{equation}

By Remark \ref{rmk:filtercolim} there exists $M \geq 0$ such that the representatives
involved in Equations \eqref{eq:colim-unitality} and \eqref{eq:colim-linearity}
become equal upon applying $\tmap_{N,N+M}$.
Let $L := N+M$ and put $r_v := B_F^M \widetilde{r}_v \in \N_0^{\reg(F)}$,
$s^{N+1+j}_v := C_F B_F^{j} \widetilde{r}_v \in \N_0^{\sink(F)}$
for each $v \in E^0$ and $0 \leq j \leq M-1$.
Set
$S^{(i)}_{v,u} := (s^{(i)}_v)_u$ and $R_{v,w} = (r_{v})_w$
for each $v \in E^0$, $w \in \reg(F)$, $u \in \sink(F)$ and $i \in \{0,\ldots, L\}$.
With this notation in place, by Lemma \ref{lem:iterating} we have
the following equality in $\Z^{V_L}$:
\begin{equation}\label{eq:unitality}
\mathfrak{u}_{L} =
\sum_{i=0}^{L} \sum_{v \in E^0} s^{(i)}_v \otimes \sigma^i
+ \sum_{v \in E^0} r_v \otimes \sigma^L.
\end{equation}
The identities of \eqref{eq:bf-mor-P} follow by comparing \eqref{eq:unitality} with \eqref{eq:colim-ordunit} applied to $F$.
By the same argument, for each $v \in \reg(E)$ we obtain
\begin{equation}\label{eq:linearity}
    \sum_{i=0}^L s^{(i)}_v \otimes \sigma^{i-1}
    + C_F r_v \otimes \sigma^L + B_F r_v \otimes \sigma^L
    = \sum_{i=0}^L \sum_{x \in E^0} (A_E)_{v,x} s^{(i)}_x \otimes \sigma^i +
    \sum_{x \in E^0} (A_E)_{v,x}r_x \otimes \sigma^L.
\end{equation}
The identities of \eqref{eq:bf-mor-S} and \eqref{eq:bf-mor-AER=RAF} follow by comparing the coefficients of $\sigma^i$ for $i=0,\dots,L$ in both sides of \eqref{eq:linearity} and taking coordinates.
\end{proof}

\begin{coro} \label{coro:reg-reg}Let $E$ and $F$ be finite graphs and
$\phi \colon \gBF(E) \to \gBF(F)$
morphism of preordered $\Z[\sigma]$-modules such that $\phi(1_E) = 1_F$.
If $E$ is regular, then so is $F$.
\end{coro}
\begin{proof} By Proposition \ref{prop:bf-map-mat}, there exists a matrix
$S^{(0)} \in \N_0^{E^0 \times \sink(F)}$ such that for each $v \in
\reg(E)$ and $u \in \sink(F)$ we have
$S^{(0)}_{v,u} = 0$ and $1 = \#E^0_{u} = \sum_{z \in E^0} S^{(0)}_{z,u}$.
Since $E = \reg(E)$, the existence of an element $u \in \sink(F)$
would imply that
$1 = \sum_{z \in E^0} S^{(0)}_{z,u} = \sum_{z \in \reg(E)} S^{(0)}_{z,u} = 0$,
a contradiction.
\end{proof}

We conclude the section by showing that isomorphisms 
between Bowen-Franks modules restrict to a bijection between 
sinks.

A vertex $v$ of a finite graph $E$ is a \emph{line-point}
if either $v$ is a sink or there exist
edges $e_1, \ldots, e_n$ such that $s(e_1) = v$, $\#s^{-1}(s(e_i)) = 1$
for all $i \in \{1, \ldots, n\}$ and $r(e_n) \in \sink(E)$. We remark
that this definition is equivalent to the one given in \cite{talented}.

\begin{rmk}\label{rmk:linept} If $E$ is a finite graph and $v \in E^0$ is a line point,
there exists $i \in \N_0$ and $u \in \sink(E)$ such that $v = \sigma^i u$ in $\gBF(E)$.
\end{rmk}

In \cite{talented}, line-points are characterized in terms of the
positive cone of the Bowen-Franks module of a graph. Recall
that an element $x \in \gBF(E)_+$ is \emph{aperiodic} if
the set $\{\sigma^i x\}_{i \in \Z}$ is infinite and \emph{minimal}
if $y \leq x$ implies $y = x$ for all $y \neq 0$. We shall also recall
from \cite[Section 2.2]{talented}
that in $\gBF(E)_+$ an element $x$ is minimal if and only if it is an \emph{atom},
meaning that if $x = z+y$ for some $z,y \geq 0$ then either $y = 0$ or $ z = 0$.

\begin{prop}[{\cite[Lemma 5.6 ii)]{talented}}] \label{prop:linept-min-ap}
Let $E$ be a finite graph.
A vertex $v \in E^0$ is a line-point if and only if it is a minimal and aperiodic
element in $\gBF(E)_+$.
\qed
\end{prop}

Before proving the proposition below, we remark that if $E$ is a finite graph,
then Lemma \ref{lem:iterating} implies that the
map $\sink(E) \to \gBF(E)$, $u \mapsto [u \otimes 1]$ is an injection.

\begin{prop} \label{prop:sink2sink}
Let $E$ and $F$ be finite graphs.
If $\phi \colon  \gBF(E) \to \gBF(F)$ is an isomorphism
of preordered $\Z[\sigma]$-modules such that $\phi(1_E) = 1_F$, then $\phi$ restricts
to a bijection $\phi \colon \sink(E) \iso \sink(F)$.
\end{prop}
\begin{proof} 
Given that $\phi$ is an isomorphism,
it follows from \eqref{eq:rk-bfg} 
that $\# \sink(E) = \rk \gBF(E) = \rk \gBF(F) = \# \sink(F)$. Since $\phi$
is injective, it will suffice to see that it restricts to a map $\sink(E) \to \sink(F)$.

Let $x \in \sink(E)$ and, using Proposition \ref{prop:bf-map-mat},
write \[
\phi(x) = \sum_{i = 0}^N
\sum_{u \in \sink(F)} S^{(i)}_{x,u} [u \otimes \sigma^i] + \sum_{w \in \reg(F)}R_{x,w} [w \otimes \sigma^N].
\]
Since $x$ is an aperiodic atom by Proposition \ref{prop:linept-min-ap}, the element $\phi(x)$
is an aperiodic atom in $\gBF(F)_+$. In particular, in the
equality above the right hand side must consist of exactly one summand and so $\phi(x) = 
[z \otimes \sigma^i]$ for some
$z \in E^0$ and $i \geq 0$. Using once again that $\phi(x) = [z \otimes \sigma^i]$ is an aperiodic atom we obtain the same conclusion for $[z \otimes 1]$; Proposition \ref{prop:linept-min-ap} then tells us that
$z$ must be a line point. Hence $[z \otimes 1] = [u \otimes \sigma^j]$
for some $u \in \sink(F)$ and $j \geq 0$, as per Remark \ref{rmk:linept}.

We have thus seen that there exists
$k \geq 0$ such that $\phi(x) = [u \otimes \sigma^k]$.
By applying the same argument to $\phi^{-1}$ and $u$, there exists
$x' \in \sink(E)$ and $k' \geq 0$ such that $\phi^{-1}(u) = [x' \otimes \sigma^{k'}]$.
From here it follows that
\[
    [x \otimes 1] = \phi^{-1}(\phi(x)) = [x' \otimes \sigma^{k+k'}],
\]
which implies $x' = x$ and $k+k' = 0$. Hence $k = k' = 0$ and
$\phi(x) = [u \otimes 1]$, as desired.
\end{proof}

\begin{rmk} \label{rmk:mat-iso-sinks}
The proof of Proposition \ref{prop:sink2sink}
says in particular 
that if $\phi \colon \gBF(E) \to \gBF(F)$ is 
a preordered $\Z[\sigma]$-module isomorphism 
mapping $1_E \mapsto 1_F$, then in the description 
of Proposition \ref{prop:bf-map-mat} we have 
$R_{u,w} = 0$, $S^i_{u,u'} = 0$ and $S^0_{u,u'} = \delta_{u,\phi(u)}$ for all 
$u \in \sink(E)$, $w \in \reg(F)$,
$u' \in \sink(F)$ and $i \in \{1, \ldots, L\}$.

\end{rmk}

\section{Lifting maps between Bowen-Franks modules to algebra maps}
\label{sec:lift}

This section will be devoted to the proof of a lifting result
concerning Bowen-Franks modules and their corresponding Leavitt path algebras. 

\begin{thm} \label{thm:gbf-lift} Let $E$ and $F$ be finite graphs. If
$\phi \colon \gBF(E)\to\gBF(F)$ is
a morphism of preordered $\Z[\sigma]$-modules such that $\phi(1_E) = 1_F$, then there exists a unital, $\Z$-graded, diagonal preserving $\ast$-homomorphism $\varphi \colon L(E) \to L(F)$ such that the following diagram commutes.
\begin{center}
\begin{tikzcd}
 K_0^{\gr}(L(E)) \arrow{r}{K_0^{\gr}(\varphi)} & K_0^{\gr}(L(F))\\
 \gBF(E) \arrow{u}{\can}\arrow{r}{\phi} & \arrow{u}[right]{\can} \gBF(F)
\end{tikzcd}
\end{center}
\end{thm}

\begin{proof} Since $L_\ell(E) = \ell \otimes_\Z L(E)$,
$L_\ell(F) = \ell \otimes_\Z L(F)$, and all $\ast$-homomorphisms
between Leavitt path algebras over $\Z$ are diagonal preserving (\cite[Corollary 5]{diago}), it 
suffices to show that there exists a unital $\Z$-graded
$\ast$-homomorphism $\varphi \colon L_\Z(E) \to L_\Z(F)$
satisfying $K_0^{\gr}(\varphi)\can = \can\phi$.

By Proposition \ref{prop:bf-map-mat}, there exist $L \in \N_0$
and matrices $S^{(0)}, \ldots, S^{(L)} \in \N_0^{E^0 \times \sink(F)}$,
$R \in \N_0^{E^0 \times \reg(F)}$ statisfying Equations
\eqref{eq:bf-mor-mat}, \eqref{eq:bf-mor-P}, \eqref{eq:bf-mor-S}, and
\eqref{eq:bf-mor-AER=RAF}.
This implies, for each $w \in \reg(F)$ and $u \in \sink(F)$, $i \in \{0, \ldots, L\}$,
the existence of partitions
\begin{equation}\label{def:Gamma-Sigma}
F^L_{w}  = \bigsqcup_{z \in E^0} \Gamma_{z,w}, \quad
F^i_{u}  = \bigsqcup_{z \in E^0} \Sigma^i_{z,u}
\end{equation}
such that $\#\Gamma_{z,w} = R_{z,w}$ and $\#\Sigma^i_{z,u} = S^{(i)}_{z,u}$.
Moreover, if $v \in \reg(E)$ then $\Sigma^0_{v,u} = \emptyset$
and there exist bijections
\begin{align}
\zeta^i_{v,u} &\colon \{(e,\beta) : e \in s^{-1}(v), \beta \in \Sigma_{r(e), u}^i\}
\iso \Sigma_{v, u}^{i+1},\qquad (0 \leq i \leq L-1),\label{bij:zetai}\\
\zeta^L_{v,u} &\colon \{(e,\beta) : e \in s^{-1}(v), \beta \in \Sigma_{r(e), u}^L\}
\iso \{\alpha f : f \in F^1, r(f) = u, \alpha \in \Gamma_{v,s(f)}\},\label{bij:zetaN}\\
\xi_{v,w} &\colon \{(e,\alpha) : e \in s^{-1}(v), \alpha \in \Gamma_{r(e), w}\}
\iso \{\alpha f : f \in F^1, r(f) = w, \alpha \in \Gamma_{v,s(f)}\}.\label{bij:xi}
\end{align}
We shall identify the images of the morphisms above with paths in $L(F)$.
We will omit the indices in the functions above since
they can be deduced from the element at which the function is being evaluated, namely $\xi(e,\alpha) = \xi_{s(e), r(\alpha)}(e,\alpha)$ and
$\zeta^i(e,\beta) = \zeta^i_{s(e), r(\beta)}(e,\beta)$.

For each $v \in E^0$ and $e \in E^1$, we define
\begin{align}
\varphi(v) &= \sum_{w\in \reg(F)}\sum_{\alpha \in \Gamma_{v,w}} \alpha \alpha^\ast
+ \sum_{u \in \sink(F)} \sum_{i=0}^L \sum_{\beta \in \Sigma^i_{v,u}}\beta\beta^\ast,
\label{def:phi-v}\\
\varphi(e) &= \sum_{w \in \reg(F)} \sum_{\alpha \in \Gamma_{r(e),w}}
\xi(e,\alpha)\alpha^\ast +
\sum_{u \in \sink(F)} \sum_{i=0}^L \sum_{\beta \in \Sigma^i_{r(e),u}}
\zeta^i(e,\beta)\beta^\ast.\label{def:phi-e}
\end{align}
We will show that
the prescriptions above
define a graded $\ast$-homomorphism.
Note that for each $v \in E^0$ and $e \in E^1$ we have
$\varphi(e) \in L(F)_1, \varphi(v) \in L(F)_0$ and, by
Lemma \ref{lem:1-paths},
\begin{align*}
\sum_{v \in E^0} \varphi(v) &= \sum_{w\in \reg(F)}\sum_{v \in E^0}\sum_{\alpha \in \Gamma_{v,w}} \alpha \alpha^\ast
+ \sum_{u \in \sink(F)} \sum_{i=0}^L \sum_{v \in E^0}\sum_{\beta \in \Sigma^i_{v,u}}\beta\beta^\ast
\\&= \sum_{w\in \reg(F)}\sum_{\alpha \in F^L_{w}} \alpha \alpha^\ast
+ \sum_{u \in \sink(F)} \sum_{i=0}^L \sum_{\beta \in F^i_{u}}\beta\beta^\ast = 1.
\end{align*}
Thus, to show that the assignments \eqref{def:phi-v} and \eqref{def:phi-e}
define a unital, graded $\ast$-homomorphism,
it suffices to verify the following relations:
\begin{align}
&\varphi(v) = \varphi(v)^\ast, &(v \in E^0)\tag{P}\label{phi:P}\\
&\varphi(v)\varphi(v') = \delta_{v,v'}\varphi(v),  &(v,v' \in E^0) \tag{V}\label{phi:V}\\
&\varphi(s(e))\varphi(e) = \varphi(e)\varphi(r(e)) = \varphi(e), &(e \in E^1)\tag{E}\label{phi:E}\\
&\varphi(g)^\ast \varphi(e) = \delta_{g,e}\varphi(r(e)),   &(g,e \in E^1)\tag{CK1}\label{phi:ck1}\\
&\varphi(v) = \sum_{e \in s^{-1}(v)} \varphi(e)\varphi(e)^\ast.  &(v \in \reg(E)) \tag{CK2}\label{phi:ck2}
\end{align}
Relations \eqref{phi:P} and \eqref{phi:V} follow directly from Lemma \ref{lem:pathsort}.
We turn our attention to \eqref{phi:E}. First, we will compute $\varphi(e)\varphi(r(e))$.
By Lemma \ref{lem:pathsort},
\begin{align*}
\varphi(e)\varphi(r(e)) &=
\sum_{w \in \reg(F)} \sum_{\alpha, \lambda \in \Gamma_{r(e),w}}
\xi(e,\alpha)\alpha^\ast\lambda\lambda^\ast +
\sum_{u \in \sink(F)} \sum_{i=0}^N
\sum_{\beta, \gamma \in \Sigma^i_{r(e),u}}
\zeta^i(e,\beta)\beta^\ast \gamma\gamma^\ast\\
&= \sum_{w \in \reg(F)} \sum_{\alpha \in \Gamma_{r(e),w}}
\xi(e,\alpha)\alpha^\ast +
\sum_{u \in \sink(F)} \sum_{i=0}^N
\sum_{\beta \in \Sigma^i_{r(e),u}}
\zeta^i(e,\beta)\beta^\ast = \varphi(e).
\end{align*}
Using Lemma \ref{lem:pathsort} once again, we see that $\varphi(s(e))\varphi(e)$
coincides with the following:
\begin{align}\label{eq:se-times-e}
    \sum_{w, w' \in \reg(F)} \sum_{\alpha \in \Gamma_{r(e),w}, \lambda \in \Gamma_{s(e),w'}}
    \lambda\lambda^\ast\xi(e,\alpha)\alpha^\ast +
    \sum_{u \in \sink(F)} \sum_{i = 0}^{L-1}
    \sum_{\beta \in \Sigma^i_{r(e),u}, \gamma \in \Sigma^{i+1}_{s(e),u}}
    \gamma\gamma^\ast\zeta^i(e,\beta)\beta^\ast\\\notag
    + \sum_{u \in \sink(F), w' \in \reg(F)} \sum_{\beta \in \Sigma^L_{r(e),u}
    \lambda \in \Gamma_{s(e),w'}} \lambda \lambda^\ast \zeta^L(e,\beta)\beta^\ast.
\end{align}
For a fixed $w \in \reg(E)$ and $\alpha \in \Gamma_{r(e), w}$, we know that
$\xi(e, \alpha) = \varepsilon f$ for some $f \in F^1$ and
$\varepsilon \in \Gamma_{s(e), s(f)}$. Therefore
if $\lambda \in \Gamma_{s(e),w'}$ we must have
\[
\lambda \lambda^\ast \xi(e, \alpha)\alpha^\ast = \delta_{\lambda, \varepsilon} \lambda f \alpha^\ast
= \delta_{\lambda, \varepsilon} \varepsilon f \alpha^\ast = \delta_{\lambda, \varepsilon} \xi(e,\alpha)\alpha^\ast.
\]
In the same fashion, fix $u \in \sink(E)$, and
$\beta \in \Sigma^i_{r(e),u}$. If $i < L$ we have that
$\gamma\gamma^\ast\zeta^i(e,\beta)\beta^\ast =
\delta_{\gamma, \zeta^i(e,\beta)} \zeta^i(e,\beta)\beta^\ast$ for each $\gamma \in \Sigma^{i+1}_{s(e),u}$.
If $i = L$, then there is a unique $\lambda \in \Gamma_{s(e),w'}$
such that $\lambda\lambda^\ast \zeta^L(e,\beta)\beta^\ast$ is nonzero, in which case it
equals $\zeta^L(e,\beta)\beta^\ast$.
Hence Equation \eqref{eq:se-times-e} agrees with
\[
\sum_{w \in \reg(F)} \sum_{\alpha \in \Gamma_{r(e),w}}
\xi(e,\alpha)\alpha^\ast +
\sum_{u \in \sink(F)} \sum_{i = 0}^L
\sum_{\beta \in \Sigma^i_{r(e),u}}
\zeta^i(e,\beta)\beta^\ast = \varphi(e)
\]
as desired.

Next we prove \eqref{phi:ck1}. In view of \eqref{phi:P}, \eqref{phi:V} and \eqref{phi:E},
we can assume without loss of generality
that $s(g) = s(e)$. Using that for each $w \in \reg(F), u \in \sink(F)$
the functions $\xi_{s(e),w}$ and $\zeta^i_{s(e),u}$ are
bijections, we obtain the following equalities:
\begin{align*}
\xi(g,\lambda)^\ast \xi(e,\alpha) &= \delta_{\xi(g,\lambda),\xi(e,\alpha)}
r(\xi(e,\alpha)) = \delta_{e,g} \delta_{\alpha, \lambda} w, & (\lambda 
 \in \Gamma_{r(g), w'}, \alpha \in \Gamma_{r(e), w})\\
\zeta^i(g,\gamma)^\ast\zeta^i(e,\beta) &=
\delta_{\zeta^i(g,\gamma),\zeta^i(e,\beta)}r(\zeta^i(e,\beta)) =
\delta_{e,g} \delta_{\beta, \gamma} u. &(\gamma \in \Sigma^i_{r(g), u'}, \beta \in \Sigma^i_{r(e), u})
\end{align*}
Consequently,
\[
\varphi(g)^\ast \varphi(e) = \sum_{w \in \reg(F)} \sum_{\alpha \in \Gamma_{r(e),w}}
\delta_{e,g}\alpha\alpha^\ast + \sum_{u\in \sink(F)} \sum_{i=0}^L
\sum_{\beta \in \Sigma^i_{r(e),u}} \delta_{e,g}\beta\beta^\ast = \delta_{e,g}\varphi(r(e)).
\]

At last, we prove \eqref{phi:ck2}. Fix $v \in \reg(E)$. By hypothesis,
we know that $\Sigma_{v,u}^0 = \emptyset$ for all $u \in \sink(F)$ and thus
\begin{align}
\varphi(v) &= \sum_{w\in \reg(F)}\sum_{\alpha \in \Gamma_{v,w}} \alpha \alpha^\ast
+ \sum_{u \in \sink(F)} \sum_{i=1}^L \sum_{\beta \in \Sigma^i_{v,u}}\beta\beta^\ast.\notag
\\&=\sum_{w\in \reg(F)}\sum_{f \in s^{-1}(w)}\sum_{\alpha \in \Gamma_{v,w}} \alpha f (\alpha f)^\ast
+ \sum_{u \in \sink(F)} \sum_{i=1}^L \sum_{\beta \in \Sigma^i_{v,u}}\beta\beta^\ast\notag\\
&= \sum_{w\in \reg(F)}\sum_{f \in s^{-1}(w)}\sum_{\alpha \in \Gamma_{v,s(f)}} \alpha f (\alpha f)^\ast
+ \sum_{u \in \sink(F)} \sum_{i=1}^L \sum_{\beta \in \Sigma^i_{v,u}}\beta\beta^\ast\notag\\
&= \sum_{f \in F^1}\sum_{\alpha \in \Gamma_{v,s(f)}} \alpha f (\alpha f)^\ast
+ \sum_{u \in \sink(F)} \sum_{i=1}^L \sum_{\beta \in \Sigma^i_{v,u}}\beta\beta^\ast.\label{phi:ck2fromv}
\end{align}
Given $e \in E^1$, a similar reasoning
as the one used to prove \eqref{phi:E} goes to show that
\begin{align*}
\varphi(e)\varphi(e)^\ast &=
\sum_{w\in \reg(F)} \sum_{\alpha \in \Gamma_{r(e),w}}
\xi(e,\alpha)\xi(e,\alpha)^\ast +
\sum_{u \in \sink(F)} \sum_{i=0}^L
\sum_{\beta \in \Sigma^i_{r(e),u}}
\zeta^i(e,\beta)\zeta^i(e,\beta)^\ast.
\end{align*}
Summing the expression above for each $e \in s^{-1}(v)$ we obtain the following:
\begin{equation}\label{phi:ck2frome}
\sum_{w\in \reg(F)} \sum_{e \in s^{-1}(v)}\sum_{\alpha \in \Gamma_{r(e),w}}
\xi(e,\alpha)\xi(e,\alpha)^\ast +
\sum_{u \in \sink(F)} \sum_{i=0}^L
\sum_{e \in s^{-1}(v)}\sum_{\beta \in \Sigma^i_{r(e),u}}
\zeta^i(e,\beta)\zeta^i(e,\beta)^\ast.
\end{equation}

Notice that given $w \in \reg(F)$ and $u \in \sink(F)$, we are summing 
over the domains of definition of the bijections $\zeta^i_{v,u}$ and $\xi_{v,u}$ respectively. From this and \eqref{phi:ck2frome} we see that
\begin{align*}
\sum_{e \in s^{-1}(v)}\varphi(e)\varphi(e)^\ast &=
\sum_{z\in F^0} \sum_{f \in r^{-1}(z)}\sum_{\alpha \in \Gamma_{v,s(f)}}
\alpha f (\alpha f)^\ast +
\sum_{u \in \sink(F)} \sum_{i=1}^L
\sum_{\beta \in \Sigma^i_{v,u}} \beta\beta^\ast\\
&= \sum_{f\in E^1}\sum_{\alpha \in \Gamma_{v,s(f)}}
\alpha f (\alpha f)^\ast +
\sum_{u \in \sink(F)} \sum_{i=1}^L
\sum_{\beta \in \Sigma^i_{v,u}} \beta\beta^\ast.
\end{align*}
The expression above is precisely \eqref{phi:ck2fromv}, which equals
$\varphi(v)$; this completes the proof of \eqref{phi:ck2}.

Finally, let us see that $K_0^{\gr}(\varphi)\can = \can \phi$. Fix $v \in E^0$.
Applying Lemma \ref{lem:pathsort} we obtain
\[
    K_0^{\gr}(\varphi)\can(v) = [\varphi(v)]
    = \sum_{w\in \reg(F)}\sum_{\alpha \in \Gamma_{v,w}} [\alpha \alpha^\ast]
    + \sum_{u \in \sink(F)} \sum_{i=0}^L \sum_{\beta \in \Sigma^i_{v,u}}[\beta\beta^\ast].
\]
If $\alpha \in \Gamma_{v,w}$ and $\beta \in \Sigma^i_{v,u}$,
by Lemma \ref{lem:paths-gbf} we have that $[\alpha\alpha^\ast] = \can([r(\alpha)_{|\alpha|}])
= \can([w_L])$ and likewise $[\beta\beta^\ast] = \can([u_i])$. Therefore
\begin{align*}
   K_0^{\gr}(\varphi)\can(v)
    &= \sum_{w\in \reg(F)}\sum_{\alpha \in \Gamma_{v,w}} \can([w_L])
    + \sum_{u \in \sink(F)} \sum_{i=0}^L \sum_{\beta \in \Sigma^i_{v,u}}\can([u_i])
    \\&= \sum_{w\in \reg(F)}\#\Gamma_{v,w} \can([w_L])
    + \sum_{u \in \sink(F)} \sum_{i=0}^L \#\Sigma^i_{v,u}\can([u_i])
    \\&= \sum_{w\in \reg(F)}R_{v,w} \can([w_L])
    + \sum_{u \in \sink(F)} \sum_{i=0}^L S^{(i)}_{v,u}\can([u_i]).
\end{align*}
The last term in the chain of equalities above agrees with 
$(\can \circ \phi)(v)$ by Equation \eqref{eq:bf-mor-mat}.
\end{proof}

\begin{ex} We go through the construction of the proof of Theorem  \ref{thm:gbf-lift}
in a concrete example. Consider the following graphs: 
\[
E =\begin{tikzcd}
\bullet_{z} \arrow["x_1"', loop, distance=2em, in=125, out=55] \arrow["x_2"', loop, distance=2em, in=305, out=235]
\end{tikzcd},
\qquad A_E = \begin{pmatrix}2\end{pmatrix},
\qquad
F =
\begin{tikzcd}
\bullet_{u} \arrow["e_1"', loop, distance=2em, in=125, out=55] 
\arrow[r, "e_2", bend left] & \bullet_{v} 
\arrow["f_1"', loop, distance=2em, in=125, out=55] \arrow[l, "f_2", bend left]
\end{tikzcd},
\qquad A_F =\begin{small}\begin{pmatrix}1&1\\1&1\end{pmatrix}\end{small}.
\]
There is a preordered $\Z[\sigma]$-module isomorphism 
$\phi \colon \gBF(E) \iso \gBF(F)$
determined 
by $1_E =[z_0] \mapsto 1_F = [u_0] + [v_0]$.
Setting $L = 0$, the matrix $R = \begin{pmatrix}1 & 1\end{pmatrix}
\in \Z^{E^0 \times F^0}$ satisfies the equations 
of Proposition \ref{prop:bf-map-mat}. Thus, there exist
partitions of paths of length $L = 0$ ending at each regular vertex, 
namely $\Gamma_{z, u} = \{u\}$, $\Gamma_{z, v} = \{v\}$, 
and bijections 
\begin{align*}
    &\xi_{z,u} \colon \{x_1, x_2\} \times \{u\} 
    \iso \{u\} \times \{e_1\} \sqcup \{v\} \times \{f_2\},\\
    &\xi_{z,v} \colon \{x_1, x_2\} \times \{v\} 
    \iso \{u\} \times \{e_2\} \sqcup \{v\} \times \{f_1\}.
\end{align*}
For example, we may set
\[
\xi_{z,u}(x_1,u) = e_1, \qquad \xi_{z,u}(x_2,u) = f_2, \qquad
\xi_{z,v}(x_1,v) = e_2, \qquad \xi_{z,v}(x_2,v) = f_1.
\]
From this choice of bijections we obtain a 
lift $\varphi \colon L_\ell(E) \to L_\ell(F)$ of $\phi$ determined 
by the following assignments:
\[
\varphi(1) = 1, \qquad \varphi(x_1) = e_1+e_2, \qquad 
\varphi(x_2) = f_1+f_2.
\]
\end{ex}

\begin{rmk} If $\phi \colon \gBF(E) \to \gBF(F)$ is an
ordered $\Z[\sigma]$-module
isomorphism mapping $1_E \mapsto 1_F$ and $\varphi \colon L(E)
\to L(F)$ is constructed as in the proof of Theorem \ref{thm:gbf-lift}, then
by Remark \ref{rmk:mat-iso-sinks} the map
$\varphi$ restricts 
to a bijection $\sink(E) \iso \sink(F)$.
\end{rmk}

\begin{rmk} Corollary \ref{coro:reg-reg} may also be derived
from Theorem \ref{thm:gbf-lift}.
Indeed, let $\phi \colon \gBF(E) \to \gBF(F)$ be a preordered $\Z[\sigma]$-module
map such that $\phi(1_E) = 1_F$.
If $E$ is regular then $L(E)$ is strongly graded, by
\cite[Theorem 3.15]{graded-str}.
Now, Theorem \ref{thm:gbf-lift} implies the
existence of a $\Z$-graded $\ast$-homomorphism $\varphi \colon L(E) \to L(F)$
and so $L(F)$ must be strongly graded as well; see \cite[Proposition 1.1.15 (4)]{hazrat-book}. Using
\cite[Theorem 3.15]{graded-str}
once again, we conclude that $F$ is regular.
\end{rmk}

We shall presently explore the question of whether there exist 
graded unital maps $\varphi \colon L_\C(E) \to L_\C(F)$
between Leavitt path algebras of finite graphs such that 
$K_0^{\gr}(\varphi)$ is an isomorphism but $\varphi$ is not.

To this end, we recall some facts regarding graph
$C^*$-algebras; we refer the reader to \cite[Section 5.2]{lpabook}
and \cite[Sections 1-4]{abramstomforde}
for further details. 
Given a graph $E$, 
its $C^*$-algebra $C^*(E)$ is given by a completion of
$L_\C(E)$ in a suitable norm (\cite[Proposition 3.1]{abramstomforde}). A $\ast$-homomorphism 
between Leavitt path $\C$-algebras can be extended upon
completion to one between the corresponding graph
$C^\ast$-algebras (\cite[Proposition 4.4]{abramstomforde}).

The algebra $C^*(E)$ is equipped with a so called \emph{gauge action} of the circle 
$S^1$ (\cite[Definition 2.13]{abramstomforde}). We can view the automorphism associated to multiplication by $z \in S^1$ as the completion of the $\ast$-automorphism given by
\[
    \mu_z \colon L_\C(E) \to L_\C(E), 
    \quad v \mapsto v, \ e \mapsto ze \qquad (v \in E^0, e \in E^1).
\]
Notice that for each homogeneous element $x \in L_\C(E)$ of degree $k$ we have $\mu_z(x) = z^k x$. Hence, a $\Z$-graded algebra $\ast$-homomorphism $L_\C(E) \to L_\C(F)$ is equivariant
with respect to the actions defined above. In particular, its completion yields an $S^1$-equivariant map 
$C^\ast(E) \to C^\ast(F)$.

Recall that 
a graph $E$ is \emph{irreducible} if for each $v,w \in E^0$
there exists a path from $v$ to $w$, and \emph{non-trivial}
if it does not consist of a single cycle. 

\begin{rmk} \label{rmk:diag-(L)}
If $E$ is an irreducible, non-trivial graph with at least one edge, 
then by \cite[Lemma 6.3.14]{lpabook} and 
\cite[Proposition 4.5 and Theorem 4.13]{core} 
it follows that $D(E)$ is a maximal commutative subalgebra of 
$L_\C(E)$. In particular, if $E$ and $F$ are irreducible, non-trivial graphs 
with at least one edge
and $\varphi \colon L_\C(E) \to L_\C(F)$ 
a diagonal preserving isomorphism, then $\varphi(D(E)) = D(F)$. In particular, 
for $\overline{D(E)}$ the closure of $D(E)$ in $C^\ast(E)$, the completion
of $\varphi$ is an $S^1$-equivariant
isomorphism which maps $\overline{D(E)}$
bijectively onto $\overline{D(F)}$.
\end{rmk}

Write $\mathcal O$ for the category of 
pointed preordered $\Z[\sigma]$-modules and $\cat{Leavitt}_\ell$
for the full subcategory of $\Z$-graded $\ell$-algebras generated 
by Leavitt path algebras. We may view the graded
Grothendieck group as a functor $K_0^{\gr}\colon \cat{Leavitt}_\ell\to\mathcal O$ sending $L(E)$
to $(K_0^{\gr}(L(E)), K_0^{\gr}(L(E))_+, [L(E)])$.

\begin{ques} \label{ques:norefle}
Do there exist finite, irreducible, non-trivial graphs $E$ 
and $F$ with at least one edge satisfying the following two conditions?
\begin{itemize}
\item[(i)] there exists an order preserving $\Z[\sigma]$-module isomorphism 
$\phi \colon\gBF(E) \xto{\cong} \gBF(F)$ 
mapping $1_{E} \mapsto 1_{F}$; 
\item[(ii)] there are no $S^1$-equivariant isomorphisms
$\varphi \colon C^\ast(E) \to C^\ast(F)$ 
such that $\varphi(\overline{D(E)}) = \overline{D(F)}$.
\end{itemize}    
\end{ques}

The motivation for Question \ref{ques:norefle} stems 
from Theorem \ref{thm:norefle} below. Conditions (i)
and (ii) are related to the notions of
\emph{shift equivalence} and \emph{strong shift equivalence}
of matrices (see e.g. \cite[Definition 1.1]{shift}).

By \cite[Theorem 7.3 and Remark 7.5]{shift}, we know that
finite, irreducible, non-trivial 
graphs whose adjacency matrices are shift equivalent 
but not strong shift equivalent satisfy all the conditions 
of Question \ref{ques:norefle} with the exception of 
the unitality requirement in condition (i). 
Such an example is that of the
graphs $E_{\mathrm{KR}}$ and $F_{\mathrm{KR}}$ with
adjacency matrices $A$ and $B$ as in the counterexample
of Kim and Roush to the irreducible case of the
Williams conjecture \cite[Section 7]{KR}. (See also 
\cite[proof of Theorem 7.4]{shift}; in loc. cit. the notation
$G_A$ and $G_B$ is employed for $E_{\mathrm{KR}}$ and $F_{\mathrm{KR}}$ respectively). This motivates the following question.

\begin{ques}\label{ques:kr}
Does there exist an ordered $\Z[\sigma]$-module 
isomorphism $\mathfrak{BF}_{\mathsf{gr}}(E_{\mathrm{KR}}) \xto{\cong} \mathfrak{BF}_{\mathsf{gr}}(F_{\mathrm{KR}})$ mapping $1_{E_{\mathrm{KR}}}$ to $1_{F_{\mathrm{KR}}}$?
\end{ques}

Note that an affirmative answer to Question \ref{ques:kr}
would, in particular, 
answer Question \ref{ques:norefle} in the affirmative.
Some partial results on Question \ref{ques:kr}
are collected in Remark \ref{rmk:kr} below.

\begin{thm} \label{thm:norefle}
If the answer to Question \ref{ques:norefle} 
is affirmative, then the functor 
$K_0^{\gr} \colon \cat{Leavitt}_\C \to \mathcal O$
does not reflect isomorphisms.
\end{thm}
\begin{proof}
We shall use that, over $\C$, the graded Grothendieck group
of a Leavitt path algebra can be identified as a pointed, preordered $\Z[\sigma]$-module
with its Bowen-Franks module.
Suppose that there exist
two finite, irreducible, non-trivial graphs $E$ and $F$ satisfying
conditions (i) and (ii) of Question \ref{ques:norefle}. 
Then, by Theorem \ref{thm:gbf-lift}, there exists a $\Z$-graded, diagonal 
preserving $\ast$-homomorphism $\varphi 
\colon L_\C(E) \to L_\C(F)$ 
such that $K_0^{\gr}(\varphi) = \phi$.
If $\varphi$ were an isomorphism, 
then Remark \ref{rmk:diag-(L)} would imply that there exists an $S^1$-equivariant 
isomorphism $C^\ast(E) \to C^\ast(F)$ 
mapping $\overline{D(E)}$ to $\overline{D(F)}$, 
contradicting (ii).
\end{proof}

\begin{rmk} \label{rmk:kr}
Consider the matrices $A$ and $B$ in \cite[Section 7]{KR}, and the matrices $R$ and $S$ defined in loc. cit. that implement their shift equivalence over $\mathbb Z$. The argument of \cite[proof of Theorem 7.3.6]{symb} 
can be used to obtain a 
shift equivalence between $A$ and $B$ over $\mathbb N_0$.
Namely, the matrices
$\tilde S := S \cdot B^6$ and $\tilde R := R \cdot A^6$
are non-negative and 
such that $A^{13} = \tilde S \tilde R$, $B^{13} = \tilde R \tilde S$, 
$\tilde R A = B \tilde R$, $A \tilde S = \tilde S B$. 
Identifying $\{1, \ldots,7\}$ with $E_{\mathrm{KR}}^0$ and $F_{\mathrm{KR}}^0$, we obtain ordered $\mathbb Z[\sigma]$-module homomorphisms 
\begin{align*}
\phi &\colon \mathfrak{BF}_{\mathsf{gr}}(E_{\mathrm{KR}}) \to \mathfrak{BF}_{\mathsf{gr}}(F_{\mathrm{KR}}),  &x \mapsto \tilde S^t \cdot x,\\
\psi &\colon \mathfrak{BF}_{\mathsf{gr}}(F_{\mathrm{KR}}) \to \mathfrak{BF}_{\mathsf{gr}}(E_{\mathrm{KR}}),
&x \mapsto \tilde R^t \cdot x.\\
\end{align*}
We recall that these maps, induced by the shift equivalence above, are isomorphisms. To see this, note that $\psi \phi$ coincides with multiplication by $(A^t)^{13}$, which can be identified with multiplication by $\sigma^{-13}$ on $\mathfrak{BF}(E_{\mathrm{KR}})$. Likewise $\phi \psi$
coincides with multiplication by $\sigma^{-13}$ on $\mathfrak{BF}(F_{\mathrm{KR}})$.

Neither of these isomorphisms is unital. To check this, 
we first observe that $A^t$ and $B^t$
are invertible, since their determinant is $-   1$. Now,
in view of Remark \ref{rmk:filtercolim},
non-unitality of $\phi$ and $\psi$ follows from
the fact that neither ${\tilde R}^t$ nor 
${\tilde S}^t$ fix the vector $v := (1,1,1,1,1,1,1)^t$.
This is because both $\tilde R$
and $\tilde S$ have a column with all entries 
greater than $1$. Moreover, a direct computation 
shows that both $u := \tilde S^t \cdot (1,1,1,1,1,1,1)^t$
and $v := \tilde R^t \cdot (1,1,1,1,1,1,1)^t$ 
have all coordinates greater than one.

One could also consider shift equivalences defined
by $S \cdot B^{6+j}$ and 
$R \cdot A^{6+j}$ for $j \in \N$. 
The induced homomorphisms are given by the matrices 
$(B^j)^t \tilde S^t$ and $(A^j)^t \tilde R^t$ respectively;
these also fail to be unital by the following argument. 

Since $E_\mathrm{KR}$ and $F_{\mathrm{KR}}$
are regular, both $(B^j)^t$ and $(A^j)^t$ are 
non-negative with no zero columns. Using once again that 
all coordiantes of $u$ and 
$v$ are greater than one, it follows that 
$(B^j)^t \tilde S^t \cdot (1,1,1,1,1,1,1)^t = (B^j)^t u$ has a coordinate
which is greater than one; the same conclusion holds for $(A^j)^t \tilde R^t \cdot (1,1,1,1,1,1,1)^t$.
\end{rmk}

\begin{rmk} We remark that, by \cite[Theorem 5.6]{arcor}, 
when $\ell$ is a field the functor $K_0^{\gr} 
\colon \cat{Leavitt}_\ell \to \cO$ does
reflect monomorphisms.
\end{rmk}

\section{Tidy maps} \label{sec:tidyequiv}

We now characterize the type of morphisms constructed 
in the course of the proof of Theorem \ref{thm:gbf-lift}.

\begin{defn}\label{defn:tidy}
We say that an algebra homomorphism $\varphi \colon L(E)\to L(F)$
is \emph{tidy} if there exist $L \geq 0$ and, for each $w \in \reg(F)$,
$u \in \sink(F)$ and $i \in \{0, \ldots, L\}$, partitions
$\{\Gamma_{v,w}\}_{v \in E^0}$ of $F^L_{w}$ and
$\{\Sigma^i_{v,u}\}_{v\in E^0}$ of $F^i_{u}$ together with bijections
\eqref{bij:zetai}, \eqref{bij:zetaN}, and \eqref{bij:xi}
such that \eqref{def:phi-v} and \eqref{def:phi-e} hold 
for all $v \in E^0$ and $e \in E^1$.
\end{defn}

We shall relate the notion of tidy homomorphism with that 
of order preserving maps as defined below.

\begin{defn}\label{def:pc} Let $E$ be a finite graph.
The \emph{positive cone} of $L_\Z(E)$ is defined to be the 
sub-*-semiring of $L_\Z(E)$ given by
\[
PC(E) = \left\{\sum_{i=1}^n 
\lambda_i \alpha_i\beta_i^\ast : \alpha_i,\beta_i \in E^\infty,
r(\alpha_i) = r(\beta_i), \ \lambda_i \in \N_0, n \in \N
\right\}.
\]
An algebra homomorphism $\varphi \colon L_\Z(E) \to L_\Z(F)$
is \emph{order preserving} if $\varphi(PC(E))
\subset PC(F)$.
\end{defn}

\begin{lem}\label{lem:DcapPE}
If $E$ is a fnite graph and $D(E)$ the diagonal
subalgebra of $L_\Z(E)$, then
\[
D(E) \cap PC(E) = \left\{\sum_{\alpha \in \cR_n(E) \sqcup \cS_n(E)} 
\lambda_\alpha \alpha\alpha^\ast : \lambda_\alpha \in \N_0, n \in \N_0 \right\}.
\]
\end{lem}
\begin{proof} Let $x \in D(E) \cap PC(E)$.
Since $x \in PC(E)$, we may write 
$x = \lambda_1 \alpha_1 \beta_1^\ast + \cdots + 
\lambda_1 \alpha_k \beta_k^\ast$ for $\alpha_i$, $\beta_i \in E^\infty$ 
such that $r(\alpha_i) = r(\beta_i)$ and coefficients $\lambda_i \in \N_0$. 
Fix $M \in \N$ for which $x \in D(E)_M$.
Applying \cite[Relation (CK2) in Definition 1.2.3]{lpabook} if necessary, we may assume that there exists $N \ge M$ 
for which $\beta_i \in \cR_N(E) \sqcup \cS_N(E)$ for all $i$. In particular we have that $x \in D(E)_N$. Writing $x$ as an $\ell$-linear combination of 
elements $\gamma\gamma^\ast$ with $\gamma \in \cR(E)_N \sqcup \cS(E)_N$, the coefficient $c_\gamma \in \Z$ 
accompanying each projection $\gamma\gamma^\ast$ 
satisfies the following equation:
\[
c_\gamma r(\gamma) = \gamma^\ast x \gamma = \left(\sum_{i=1}^k \delta_{\gamma, \alpha_i}\delta_{\gamma, \beta_i} \lambda_i\right) r(\gamma).
\]
Thus, it follows that $c_\gamma \ge 0$; this completes the proof.
\end{proof}

\begin{thm}\label{thm:tidy-char} 
Let $E$ and $F$ be finite graphs and
$\varphi \colon L(E) \to L(F)$
a unital $\ell$-algebra homomorphism. The following 
statements are equivalent.
\begin{itemize}
    \item[i)] The morphism $\varphi$ is tidy.
    \item[ii)] The morphism $\varphi$ is the scalar extension along $\ell$ of a unital, order preserving, $\Z$-graded $\ast$-homomorphism $L_\Z(E) \to L_\Z(F)$.
\end{itemize}
\end{thm}
\begin{proof} Implication $i) \Rightarrow ii)$ follows from the proof of Theorem \ref{thm:gbf-lift}, 
we prove the converse. 
We recall that by \cite[Corollary 5]{diago} 
all $\ast$-homomorphisms between Leavitt
path $\Z$-algebras are diagonal preserving.
Since $\varphi$ is the scalar extension 
of a graded and order preserving homomorphism $L_\Z(E) \to L_\Z(F)$, 
for each $e \in E^1$ there must exist
paths $\alpha_{e,1}, \ldots, \alpha_{e,n_e}$, 
$\beta_{e,1}, \ldots, \beta_{e,n_e}$ in $F$ 
and scalars $\lambda_{e,1},\ldots, \lambda_{e,n_e} \in \N_0$
such that $r(\alpha_{e,i}) = r(\beta_{e,i})$, $|\alpha_{e,i}| = 1 +|\beta_{e,i}|$ and 
$\varphi(e) = \sum_{i=1}^{n_e} \lambda_{e,i} \alpha_{e,i} \beta_{e,i}^\ast$. 
By a repeated 
application of \cite[Relation (CK2) in Definition 1.2.3]{lpabook}, 
there exists $N \in \N_0$ such that for all $e \in E^1$, we have 
$\beta_{e,i} \in \cR_{N}(F) \sqcup \cS_{N}(F)$. 

Since $E^0$ is finite and contained in
$D(E) \cap PC(E)$, increasing $N$ if necessary,
by Lemma \ref{lem:DcapPE} we may write
\[
\varphi(v) = \sum_{\alpha \in \cR_{N}(F)} \lambda_{v,\alpha} \alpha\alpha^\ast + 
\sum_{\beta \in \cS_N(F)}\lambda_{v,\beta} \beta\beta^\ast, 
\qquad (\lambda_{v,\alpha}, \lambda_{v,\beta} \in \N_0)   
\]
for all $v \in E^0$.
By the fact that $\varphi(v)$ is an idempotent together with Lemma \ref{lem:pathsort}, 
we obtain that the coefficients $\lambda_{v,\alpha}$ and $\lambda_{v,\beta}$ must lie in $\mathrm{Idem}(\Z) = \{0,1\}$. 

Put $\Gamma_{v,w} = \{\alpha \in E^N_{w} : \lambda_{v,\alpha} = 1\}$ for each $w \in \reg(E)$
and $\Sigma^i_{v,u} = \{\beta \in E^i_{u} : \lambda_{v,\beta} = 1\}$ for each $u \in \sink(E)$,
$i \in \{1,\ldots, N\}$. With this notation in place, 
\begin{equation}\label{conseq:phi-v}
    \varphi(v) = \sum_{w \in \reg(E)}\sum_{\alpha \in \Gamma_{v,w}}\alpha\alpha^\ast + 
\sum_{u \in \sink(E)} \sum_{i=0}^N \sum_{\beta \in \Sigma^i_{v,u}}\beta\beta^\ast.   
\end{equation}
Next we show that the sets $\Gamma_{v,w}$ and $\Sigma_{v,u}^i$ define the desired partitions.
Given two distinct vertices $v \neq v'$,
\[
0 = \varphi(v)\varphi(v') = 
\sum_{w \in \reg(E)}\sum_{\alpha \in \Gamma_{v,w} \cap \Gamma_{v',w}}\alpha\alpha^\ast + 
\sum_{u \in \sink(E)} \sum_{i=0}^N 
\sum_{\beta \in \Sigma^i_{v,u} \cap \Sigma^i_{v',u}}\beta\beta^\ast,
\]
and thus $\Gamma_{v,w} \cap \Gamma_{v',w} = \emptyset$ and $\Sigma^i_{v,u} \cap \Sigma_{v',u}^i = \emptyset$.
Now, by unitality of $\varphi$, 
\begin{align*}
1 = \sum_{v \in E^0} \varphi(v) &= 
\sum_{v \in E^0}\sum_{w \in \reg(E)}\sum_{\alpha \in \Gamma_{v,w}}\alpha\alpha^\ast + 
\sum_{v \in E^0}\sum_{u \in \sink(E)} \sum_{i=0}^N 
\sum_{\beta \in \Sigma^i_{v,u}}\beta\beta^\ast \\
&=\sum_{w \in \reg(E)}\sum_{\alpha \in \bigcup_{v \in E^0}\Gamma_{v,w}}\alpha\alpha^\ast + 
\sum_{u \in \sink(E)} \sum_{i=0}^N 
\sum_{\beta \in \bigcup_{v \in E^0}\Sigma^i_{v,u}}\beta\beta^\ast.
\end{align*}
Thus $\mathcal R_N(F) = \bigcup_{v \in E^0,w \in \reg(E)} \Gamma_{v,w}$ and 
$\cS_N(F) = \bigcup_{v \in E^0 , u \in \sink(E) , i \in \{0,\ldots,N\}} \Sigma^i_{v,u}$.
By intersecting these equalities 
with $E^N_{w}$ and $E^i_{u}$ we obtain that 
$E^N_{w} = \bigsqcup_{v \in E^0} \Gamma_{v,w}$ 
and $E^i_{u} = \bigsqcup_{v \in E^0} \Sigma^i_{v,u}$
respectively.

Fix $e \in E^1$. By our previous observations, 
we may write
\[
\varphi(e) = 
\sum_{\alpha \in \cR_{N}(F)} x_\alpha \alpha^\ast + \sum_{\beta\in\cS_N(F)}\sum_{i=0}^N y_\beta \beta^\ast
\]
with $x_\alpha$ a finite sum of paths of length $N+1$ 
with range $r(\alpha)$ and 
$y_\beta$ a finite sum of paths of length $|\beta|+1$
with range $r(\beta)$.
Moreover, since $\varphi(e) = \varphi(e)\varphi(r(e))$, by  \eqref{conseq:phi-v} 
we get 
\[
\varphi(e) = \sum_{w \in \reg(E)} 
\sum_{\alpha \in \Gamma_{r(e),w}} x_\alpha \alpha^\ast + \sum_{u \in \sink(E)} 
\sum_{i=0}^N\sum_{\beta \in \Sigma^i_{r(e),u}} y_\beta \beta^\ast.
\]
For each $\alpha, \beta$ as above, write 
$x_\alpha = \sum_{j=0}^{n_\alpha}c_{\alpha,j} \gamma_{\alpha,j}$ 
with $n_\alpha \ge 0$, $c_{\alpha,j} \in \N$, and distinct paths
$\gamma_{\alpha,j} \in E_{r(\alpha)}^{N+1}$. Likewise write 
$y_\beta = \sum_{j=0}^{n_\beta}c_{\beta,j} \gamma_{\beta,j}$ 
with $n_\beta \ge 0$, 
$\gamma_{\beta, j} \in E_{r(\beta)}^{|\beta|+1}$, $c_{\beta,j} \in \N$. Then 
\[
x_\alpha^\ast x_{\alpha'} = \left(\sum_{i \in [n_\alpha], j \in [n_{\alpha'}]}
\delta_{\gamma_{\alpha,i}, \gamma_{\alpha',j}}c_{\alpha,i}c_{\alpha',j}\right) r(\alpha)
\text{ and } 
y_\beta^\ast y_{\beta'} = \left(\sum_{i \in [n_\beta], j \in [n_{\beta'}]}
\delta_{\gamma_{\beta,i}, \gamma_{\beta',j}}c_{\beta,i}c_{\beta',j}\right) r(\beta).
\]
Since 
\[
\varphi(r(e)) = \varphi(e)^\ast\varphi(e)
= \sum_{w \in \reg(E)} 
\sum_{\alpha,\alpha' \in \Gamma_{r(e),w}} \alpha x_\alpha^\ast x_{\alpha'} 
\alpha'^\ast + \sum_{u \in \sink(E)} 
\sum_{i=0}^N\sum_{\beta,\beta' \in \Sigma^i_{r(e),u}} \beta 
y_\beta^\ast y_{\beta'} \beta'^\ast, 
\]
necessarily
$n_{\alpha} = n_{\beta} = 1$ 
and $c_{\alpha,1} = c_{\beta,1} = 1$ for all $\alpha, \beta$.
Hence $x_\alpha, y_\beta \in E^\infty$ and moreover 
the maps $\alpha \mapsto x_\alpha$
and $\beta \mapsto y_\beta$ have to be injective,
since $x_\alpha^\ast x_{\alpha'}$ and $y_{\beta}^\ast y_{\beta'}$
ought to be zero for $\alpha \neq \alpha'$ and $\beta \neq \beta'$.

We can now define $\xi(e,\alpha) := x_\alpha$ for each $\alpha \in \Gamma_{r(e),w}$
and $\zeta^i(e,\beta) := y_\beta$ for each $\beta \in \Sigma^i_{r(e),u}$, so that 
\[
\varphi(e) = \sum_{w \in \reg(E)} 
\sum_{\alpha \in \Gamma_{r(e),w}} \xi(e,\alpha) \alpha^\ast + \sum_{u \in \sink(E)} 
\sum_{i=0}^N\sum_{\beta \in \Sigma^i_{r(e),u}} \zeta^i(e,\beta) \beta^\ast.    
\]
Notice that $\varphi(s(e))\xi(e,\alpha)\alpha^\ast$ will be zero 
if $\xi(e,\alpha)$ does not start at a path of length $N$ belonging to 
$\bigsqcup_{w \in E^0} \Gamma_{s(e),w}$,
and it will coincide with $\xi(e,\alpha)\alpha^\ast$ otherwise. 
Since $\varphi(s(e))\varphi(e) = \varphi(e)$, this establishes that
$\xi(e,\alpha) \in \bigsqcup_{w \in \reg(F)} \Gamma_{s(e),w} \times F_{w,r(\alpha)}$
and in a similar fashion that 
$\zeta^N(e,\beta) \in
\bigsqcup_{w \in \reg(F)} \Gamma_{s(e),w} \times F_{w,r(\beta)}$ and 
$\zeta^i( e,\beta) \in \Sigma^{i+1}_{s(e),r(\beta)}$
when $|\beta| < N$.

The proof will be concluded once we see that
for vertices $v \in \reg(E), w \in \reg(F), u \in \sink(F)$
the maps  
\begin{align*}
    \xi_{v,w} &\colon \bigsqcup_{x \in E^0} E_{v,x} \times \Gamma_{x,w} \to 
    \bigsqcup_{z \in \reg(F)} \Gamma_{v,z} \times F_{z,w},\\
    \zeta_{v,u}^N &\colon \bigsqcup_{x \in E^0} 
    E_{v,x} \times \Sigma^N_{x,u} \to
    \bigsqcup_{z \in \reg(F)} \Gamma_{v,z}\times F_{z,u}\\
    \zeta_{v,u}^i &\colon \bigsqcup_{x \in E^0} 
    E_{v,x} \times \Sigma^{i}_{x,u} \to \Sigma_{v,u}^{i+1}, 
    \qquad (0 \leq i \leq N-1). 
\end{align*}
are bijective.

We have already observed that for a fixed $e \in E^1$, the assignments
$\alpha \mapsto \xi_{s(e),r(\alpha)}(e,\alpha)$ and 
$\beta \mapsto \zeta_{s(e),r(\beta)}^{|\beta|}(e,\beta)$
are injective. The fact that if $e,f \in E^1$ are distinct edges starting at $v$
and $\alpha \in \Gamma_{x,w}, \alpha' \in \Gamma_{y,w}$
then $\xi_{v,w}(e,\alpha) \neq \xi_{v,w}(f,\alpha')$ follows from 
the equation $\varphi(f)^\ast \varphi(e) = 0$. 
The same argument applies to prove the injectivity of
each function $\zeta_{v,w}^i$. 

To conclude we show surjectivity.
Note that by \cite[Relation (CK2) in Definition 1.2.3]{lpabook}
and \eqref{conseq:phi-v}, the element $\varphi(v)$ must coincide with 
the sum of all expressions $\gamma f(\gamma f)^\ast$ 
with $\gamma f \in \bigsqcup_{z \in E^0, z \in \reg(F)} \Gamma_{v,z}\times F_{z,z'}$ 
and $\delta\delta^\ast$
with $\delta \in \bigcup_{i=0}^N \bigcup_{u \in \sink(F)}\Sigma_{v,u}^i$. 
At the same time 
$\varphi(v) = \sum_{s(e) = v}\varphi(e)\varphi(e)^\ast$
must coincide with the sum of expressions
$\xi(e,\alpha)\xi(e,\alpha)^\ast$ and $\zeta^i(e,\beta)\zeta^i(e,\beta)^\ast$ for 
each $i$, $\alpha \in \cR(F)_N$, $\beta \in \cS(F)_N$. This together with restrictions of the respective codomains implies that each map is surjective.
\end{proof}

\begin{ex} If $E = \cR_1$ is the graph with one vertex and 
one loop, then $L(E) \simeq \ell[t,t^{-1}]$ where $t$ has degree $1$
and $t^\ast = t^{-1}$.
For each $N \ge 0$, there is only one path of length $N$ in $E$, which corresponds under the isomorphism above to the monomial $t^N$. 
Thus the only tidy map $f \colon L(E) \to L(E)$ is the identity. 
In particular, there exist unital $\Z$-graded, diagonal preserving, 
involution preserving maps that are not tidy; for example, 
the one determined by $t \mapsto -t$.
\end{ex}

\begin{coro} \label{coro:tidy-comp} The composite 
of two tidy homomorphisms is again tidy.
\end{coro}
\begin{proof} The conditions of ii) in Theorem \ref{thm:tidy-char} are preserved by composition.
\end{proof}

\bibliographystyle{abbrvnat} 
\bibliography{main}

\end{document}